\documentclass[times]{oupau_preprint}
\usepackage{amsmath,amsthm,amssymb}
\usepackage{listings,graphicx,amsmath,varioref,amscd,color,bm,stmaryrd,lipsum}
\usepackage{times}
\usepackage{enumerate}
\usepackage{latexsym}
\usepackage{fancybox}
\usepackage{comment}
\usepackage{wasysym}
\usepackage{cancel}
\usepackage{ulem}
\usepackage{epsfig}
\usepackage{float}
\usepackage{amsfonts}
\usepackage{srcltx}
\usepackage{esint}
\usepackage[dvipsnames]{xcolor}
\usepackage{enumerate}
\DeclareMathOperator{\arctanh}{arctanh}
\usepackage{esint}

\newcommand{\ut}{{\underline t}}
\newcommand{\ux}{{\underline x}}
\newcommand{\uQ}{{\underline Q}}

\def\uu{\underline{u}}
\def\ou{\overline{u}}
\def\re{\mathbb{R}}

\def\cal{\mathcal}
\def\bk{\color{black}}
\def\gr{\color{Rubingreen}}
\def\bl{\color{blue}}

\def\elleuu{\ell_{\vare, \underline{u}}}
\def\elleu{\ell_{\vare, {u}}}
\newcommand{\und}{\underline}

\def\intt{\int_{0}^{\tau}}
\def\into{\int_{\Omega}}

\def\rn{\mathbb{R}^{N}}
  
 \def\w{{\bf w}} \def\z{{\bf z}} 
 \def\vare{\varepsilon}
\DeclareMathOperator{\dive}{div}
\DeclareMathOperator{\supp}{supp}

\DeclareMathOperator*{\sign}{sign}

\DeclareMathOperator*{\essinf}{ess\,inf}

\renewcommand{\epsilon}{{\varepsilon}}

\def\vp{\varphi}

 \newcommand{\res}               {\!\!\mathop{\hbox{
                                \vrule height 7pt width .5pt depth 0pt
                                \vrule height .5pt width 6pt depth 0pt}}
                                \nolimits}

\newcommand{\R}{{\mathbb R}}

\newcommand{\N}{{\mathbb N}}

\renewcommand{\d}{\,{\mathrm d}}
 \newcommand{\eps}{{\varepsilon}}
 \def\1{\raisebox{2pt}{\rm{$\chi$}}}
\def\dys{\displaystyle}
\newcommand{\matdot}{}

\newcommand{\ignore}[1]{}

\newcommand{\blue}[1]{{#1}}
\newcommand{\black}[1]{{#1}}
\def\bk{\black}

\newtheorem{thm}{Theorem}[section]
\newtheorem{remark}[theorem]{Remark}
\renewcommand{\gr}[1]{{\color{OliveGreen}#1}}

\begin{document}

\shorttitle{Nonlinear diffusion in transparent media}

\title{Nonlinear diffusion in transparent media}

\author{Lorenzo Giacomelli\affil{1}, Salvador Moll\affil{2}, and Francesco Petitta\affil{1}}
\abbrevauthor{L. Giacomelli, S. Moll, and F. Petitta}
\headabbrevauthor{Giacomelli, L., Moll, S., and Petitta, F.}

\address{\affilnum{1}SBAI Department,
 Sapienza University of Rome, Via Scarpa 16, 00161 Roma, Italy and
 \affilnum{2} Departament d'An\`{a}lisi Matem\`atica,
Universitat de Val\`encia, C/ Dr. Moliner, 50, 46100 Burjassot, Spain}
\correspdetails{
francesco.petitta@sbai.uniroma1.it}

\received{}
\communicated{}

\begin{abstract}
{We consider a prototypical nonlinear parabolic equation whose flux has three distinguished features: it is nonlinear with respect to both the unknown and its gradient, it is homogeneous, and it depends only on the direction of the gradient. For such equation,} we obtain existence and uniqueness of entropy solutions to the Dirichlet {problem, the homogeneous Neumann problem, and the Cauchy problem}. Qualitative properties of solutions, such as finite speed of propagation and {the occurrence of} waiting-time phenomena, with sharp bounds, are shown. We also {discuss the formation of jump} discontinuities both at the boundary of the {solutions'} support {and} in the bulk.


\newline
{\noindent \bf MSC Classification [2020]} {35K20, 35D99, 35K67, 35L65}
\end{abstract}

\maketitle

\tableofcontents

\section{Introduction}


%
%
%

{The paper is concerned with the following PDE:}
\begin{equation}\label{m}
u_t=  \dive \left(u^m \frac{ \nabla u}{|\nabla u|}\right){.}
\end{equation}
%
%
Two particular {values} 
of the parameter $m$ lead to well known equations. When $m=0$, {\eqref{m}} 
coincides with the {\it total variation flow}: we refer to the monograph \cite{ACMBook} for a detailed study of the subject
and to \cite{SapiroBook} for its applications in image processing. The case $m=1$ (the so-called {\it {heat equation in} transparent media}) was considered in \cite{ACMM-JEE}, where  existence and uniqueness of entropy solutions to the Cauchy problem for  \eqref{m} were obtained. In addition, the authors showed that solutions to the {\it relativistic heat equation},
\begin{equation}\label{ree}
\frac{\partial u}{\partial t}=\varrho\dive \left(u\frac{\nabla u}{\sqrt{u^2+\varrho^2|\nabla u|^2}}\right),
\end{equation}
converge to solutions of \eqref{m} (with $m=1$) as $\varrho\to +\infty$.

\smallskip

{Our focus is on the case $m>1$, in which} \eqref{m} is the formal limit of the {\it relativistic porous medium equation},
\begin{equation}\label{pmrhe}
\frac{\partial u}{\partial t}=\varrho\dive \left(\frac{{u^m}\nabla u}{\sqrt{u^2+\varrho^2|\nabla u|^2}}\right), \quad m>1\,,
\end{equation}
as the kinematic viscosity $\varrho$ tends to $+\infty$ (here the maximal speed of propagation has been normalized to $1$). Eq. \eqref{pmrhe} was introduced in \cite{r1,r2} while studying heat diffusion in neutral gases (precisely with $m=3/2$). Existence and uniqueness of solutions {to} the Cauchy problem associated to {\eqref{pmrhe}} were obtained in \cite{ACM_arma05}. This equation has received recently some attention and different key-features of solutions, such as propagation of support, waiting time phenomena, speed of discontinuity fronts, and pattern formations, have been addressed by many authors \cite{ACM_jde08,Caselles_jde11,G15,GMP1, Calvo_siam,cccss2, CCCSS_in16}.

\smallskip

Our interest in Eq. \eqref{m} is twofold.

\smallskip

{\it{Shock formation.}}. First of all, the dynamics of shock formation for solutions to \eqref{pmrhe} is not yet fully understood in this type of parabolic equations with hyperbolic phenomena. The studies are limited to some equations related to \eqref{pmrhe} in the pioneering contributions \cite{bdp,bl2, bl3}  and  to numerical simulations \cite{ACMSV_siam12,CCM_plms13}. Since \eqref{m} and \eqref{pmrhe} formally coincide where $|\nabla u|\gg 1$, in particular at  discontinuity fronts, \eqref{m} could serve as a prototype equation for investigating such phenomena. Moreover, Eq. \eqref{m} {has two scaling invariances}: thus one can expect to clarify and study qualitatively the strong interplays between hyperbolic and parabolic mechanisms in this type of flux--limited diffusion equations.

\smallskip

{\it{Well-posedness}}. \eqref{m} stands as a model for autonomous evolution equations in divergence form which, though of second order, have the same scaling {as that} of a first order nonlinear conservation law. For this type of equations, a well-posedness theory is not known at our best knowledge.



\medskip


Concerning {well-posedness, we will consider the Dirichlet problem, the homogeneous Neumann problem (both in a bounded domain $\Omega$), and the Cauchy problem. Our arguments rely on} nonlinear semigroup theory.
{In a bounded domain $\Omega$, in \cite{GMP} we studied the resolvent equation of \eqref{m}, i.e.}
%
%
\begin{equation}\label{pde-a}
u-f=\dive\left(u^m\frac{\nabla u}{|\nabla u|}\right) \quad \mbox{in }\ \Omega,
\end{equation}
and we obtained existence and contraction in $L^1(\Omega)$ (see Theorem \ref{thm-elliptic} below).
{By associating  {an} $m-$accretive operator in $L^1(\Omega)$ to solutions to \eqref{pde-a} we obtain existence of a mild solution to \eqref{m}.}
%
%
In order to characterize such solution, we introduce a definition of entropy solutions and subsolutions to \eqref{m} and we prove that the semigroup solution is in fact an entropy solution. Finally, we show that a comparison principle holds in $L^1$ between subsolutions and solutions, which yields uniqueness of solutions.  This programme is worked out in Section 3 for the nonhomogeneous Dirichlet problem associated to \eqref{m}, while the corresponding results for the homogeneous Neumann problem {and the Cauchy problem are} discussed in Section {4 and 6, respectively}.

\medskip

The second main objective of this paper is to study qualitative properties of solutions to \eqref{m}. {In Section 5,} we construct a family of {compactly supported} self-similar ${S}BV$-solutions{:} 
together with the comparison principle, this permits to show the finite speed of propagation property.
{In Section 6,}
thanks to the finite speed of propagation property, we obtain existence {and uniqueness} of solutions to the Cauchy problem for bounded and compactly supported initial data{. There,} we also characterize entropy solutions as those distributional solutions that satisfy the corresponding Rankine-Hugoniot jump conditions (together with an inequality for the Cantor part, if any). {In Section 7 we perform} a complete study of the waiting time phenomenon: we show that there is a scaling-wise sharp bound on the behavior at the boundary of the {solutions'} support, {which discriminates between occurrence and non-occurrence of a waiting time phenomenon}. The corresponding results {for Eq. \eqref{pmrhe}} are contained in \cite{G15, GMP1}. Finally, in the one-dimensional case we discuss similarities and differences between the behavior of solutions to \eqref{m} and those of the Burger's equation. 
{This is done in Section 8}, where we also show that {the formation of jump discontinuities may take place} both at the boundary of the {solution's} support {and} in the bulk.

\section{Preliminaries and Notation}\label{ss-not}

{
Throughout the paper, $m>1$ and $\Omega$ is a bounded open subset of $\R^N$ with Lipschitz boundary $\partial\Omega$.
}
For a general $\ell\in L^{1}_{loc} (\mathbb R)$, we let
\begin{equation}
J_{\ell}(s)=\int_{0}^{s}\ell(\sigma)\ d\sigma\,,\ \ \text{and}\ \ \ \Phi_{\ell}(s)=\int_{0}^{s} \ell'(\sigma)\vp (\sigma)\ d\sigma\,,
\end{equation}
where we have written $\varphi(s):=s^m$, for $s>0$ to ease the notation.
Moreover, let
 \begin{eqnarray*}
& \mathcal{L} = \{\ell:[0,\infty)\to [0,\infty):\ \ell '\geq 0,\ {\rm Lip}(\ell)<\infty,\ {\ell(0)=0,} \ {\rm supp}({\ell'})\subset (0,\infty{ )}\}\,.
\end{eqnarray*}
For $a,b,l\in [-\infty,+\infty]$ we let $$T_{a,b}^l(r)=\max\{\min\{b,r\},a\}-l,$$ and we define, for $r\geq0$,
\begin{eqnarray*}
&\mathcal T_+ = \{T_{a,b}^l: \ 0<a<b, \ l\le a\}.
\end{eqnarray*}
For a given $T=T_{a,b}^l \in \mathcal T_+$, we let $T^0:= T + l = T_{a,b}^0.$
{The subscript $+$ on a function space denotes that the functions within it are nonnegative.}

\begin{figure}[htbp]\centering
\includegraphics[width=2in]{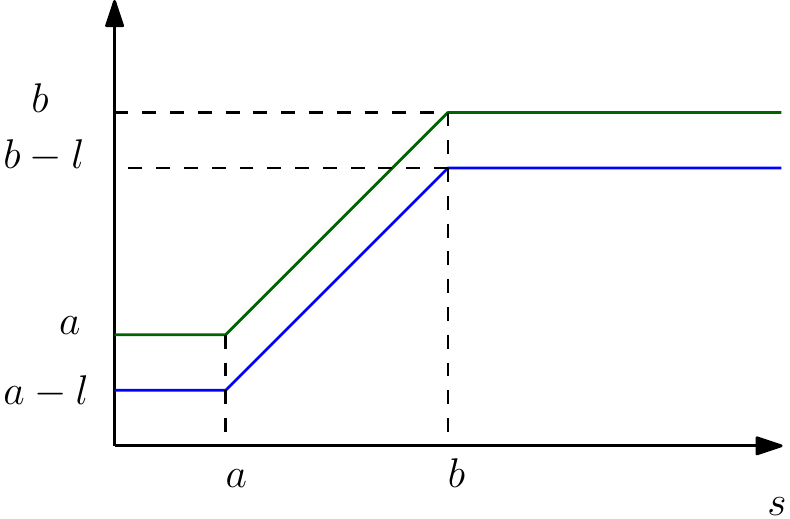}
\caption{$\bl T_{a,b}^{l}(s)$ and $\ \gr{T_{a,b}^{0}(s)}$}
\end{figure}

We denote by $\mathcal H^{N-1}$ the $(N-1)$-dimensional Hausdorff measure, by $\mathcal L^N$ the $N$-dimensional Lebesgue measure, and by ${\mathcal M}(\Omega)$ the space of  finite Radon measures on $\Omega$ (see \cite[Def. 1.40]{AFPBook}). The subscript $_0$ denotes spaces of compactly supported functions. We recall that ${\mathcal M}(\Omega)$ is the dual space of $C_0(\Omega)$. We let $\mathcal D(\Omega):=C_0^\infty(\Omega)$ and  $\mathcal D'(\Omega)$ its dual.

\smallskip

When no ambiguity arises, we shall often make use of the simplified notation $\|v\|_{q}$, $1\leq q\leq \infty$ to indicate the Lebesgue norms of $v$; here $v$ can be either a scalar function in  $L^q(\Omega)$ or a vector field in $(L^q(\Omega))^N$ (usually indicated by ${\bf v}$). From time to time we will also use the following  notation:
$$
\int_\Omega f(x) \d x:= \int_\Omega f\,.
$$

\subsection{The space $L^{\infty}_{loc,w}((0,\tau];\mathcal{M}(\Omega))$}

For $\tau\in (0,+\infty]$ we denote by $L^{\infty}_{loc,w}((0,\tau];\mathcal{M}(\Omega))$ the set of measures $\mu\in \mathcal{M}(Q_\tau)$  for which for  a.e.~$t\in (0,\tau)$ there is a measure $\mu(\cdot,t)\in \mathcal{M}(\Omega)$  such that:

\begin{itemize}

\item[$(i)$] for all $\zeta\in C_c(Q_\tau)$ the map $t\mapsto \left\langle \mu(\cdot,t),\zeta(\cdot,t)\right\rangle_{\Omega}$ belongs to $L^1(0,\tau)$ and
\begin{equation}\label{disiU}
\left\langle \mu,\zeta\right\rangle_{Q_\tau}=\int_0^{\tau}\left\langle \mu(\cdot,t),\zeta(\cdot,t)\right\rangle_{\Omega}\,dt\,;
\end{equation}

\item[$(ii)$] the map $t\mapsto \|\mu(t)\|_{\mathcal{M}(\Omega)}$ belongs to $L^{\infty}_{loc}((0,\tau])$.
\end{itemize}
Accordingly, for $0<\overline\tau<\tau$, we use the notation
$$
\|\mu\|_{L^{\infty}_{w}([\overline\tau,\tau];\mathcal{M}(\Omega))}:=\mbox{ess}\sup_{t\in (\overline\tau,\tau)}\|\mu(t)\|_{\mathcal{M}(\Omega)}
\quad \text{for  $\mu\in L^{\infty}_{loc,w}((0,\tau];\mathcal{M}(\Omega))$\,.}
$$
Observe that by the above definition the map  $t\mapsto \left\langle \mu(\cdot,t),\rho\right\rangle_{\Omega}$
is measurable for all  $\rho\in C_c(\Omega)$, thus the map  $(0,\tau)\ni t\mapsto \mu(t)\in  \mathcal{M}(\Omega)$ is weakly* measurable.

\subsection{$TBV$-functions}

We use standard notations and concepts for $BV$ functions as in \cite{AFPBook}; in particular, for $u\in BV(\R^N)$, $\nabla u\mathcal L^N$, resp.  $D^s u$, denote  the absolutely continuous, resp. singular, parts of $Du$ with respect to the Lebesgue measure $\mathcal L^N$, {$\tilde \nabla u$ denotes the diffuse part of $Du$; i.e. $\tilde \nabla u:=\nabla u\mathcal L^N+D^c u$, with $D^c u$ is the Cantor part of $Du$,}  $J_u$ denotes its jump set. For any $BV$-function $u$, we denote $[u]:=u^+-u^-$ on $J_u$. From now on, we will always identify a $BV$-function with its precise representative.

Let
\begin{equation*}
TBV_+(\Omega) = \{u\in L^1(\Omega;[0,+\infty)): \ T_{a,\infty}^0(u)\in BV(\Omega) \ \mbox{for all } a>0 \}.
\end{equation*}
Given $u\in L^{1}_{loc}(\Omega)$, the upper and lower approximate limits of $u$ at a point $x\in\Omega$ are defined respectively as
\begin{eqnarray*}
  u^\vee(x)&:=& \inf\{t\in \R : \lim_{\rho\downarrow 0}\rho^{-N} |\{u>t\}\cap B_\rho(x)|=0\},\\
  u^\wedge(x)&:=& \sup\{t\in \R : \lim_{\rho\downarrow 0}\rho^{-N} |\{u<t\}\cap B_\rho(x)|=0\}.
\end{eqnarray*}
We let $S_u^*:=\{x\in\Omega : u^\wedge(x)<u^\vee(x)\}$ and
\begin{equation}\label{dbtv}
DTBV_+(\Omega)=\{u\in TBV_+(\Omega): \ \mathcal H^{N-1}(S_u^*)=0\}.
\end{equation}
The set of weak approximate jump points is the subset $J_u^*$ of $S_u^*$ such that there exists a unit vector $\nu_u^*(x)\in\R^N$ such that the weak approximate limit of the restriction of $u$ to the hyperplane $H^+:=\{y\in \Omega: \langle y-x,\nu_u^*(x)\rangle>0\}$ is $u^\vee(x)$ and the weak approximate limit of the restriction of $u$ to  $H^-:=\{y\in \Omega: \langle y-x,\nu_u^*(x)\rangle<0\}$ is $u^\wedge(x)$. In \cite[Page 237]{AFPBook} it is shown that for any $u\in L^1_{loc}(\Omega)$,  $J_u\subset J_u^*$. Moreover, $u^\vee(x)=\max\{u^+(x),u^-(x)\}$, $u^{\wedge}(x)=\min\{u^+(x),u^-(x)\}$ and $\nu_u^*(x)=\pm \nu_u(x)$ for any $x\in J_u$.  Furthermore, (\cite[Lemma 2.1]{GMP}  $S_u^*$ is countably $\mathcal H^{N-1}$ rectifiable and $\mathcal H^{N-1}(S_u^*\setminus J_u^*)=0$.

Finally, $TBV_+(\Omega)$ functions have a well defined trace on the boundary $\partial\Omega$ (see \cite[Lemma 5.1]{GMP}.


Given $u\in TBV_{+}(\Omega)$ we use the following notation for consistency with previous works, e.g. \cite{ACM, ACM_arma05,ACM_jde08,ACMM-JEE,ACMM-MA,ACMM-MA-E,GMP1}: 
$$
h (u, D\ell (u))=|D \Phi_{\ell}(u)|\,.
$$

\subsection{Divergence-measure vector-fields}\label{GreenAnz}

%
%
{We define the space}
$$
 X_{\mathcal{M}}(\Omega) = \left\{ \z \in L^{\infty}(\Omega; \R^N) \ :
 \, \dive \z \in {\mathcal M}(\Omega) \right\}.
$$
In \cite[Theorem 1.2]{Anz_ampa83} (see also \cite{ACMBook,ChF_arma99}), the weak trace on $\partial \Omega$ of the normal component of $\z \in X_{\mathcal{M}}(\Omega)$ is defined as a linear operator $[\cdot,\nu^\Omega]: X_{\mathcal{M}}(\Omega) \rightarrow L^\infty(\partial \Omega)$ such that
$
\Vert \, [\z,\nu^\Omega] \, \Vert_{L^\infty(\partial \Omega)} \leq \Vert \z \Vert_{\infty}
$
for all $\z \in X_{\mathcal{M}}(\Omega)$ and $[\z,\nu^\Omega]$ coincides with the point-wise trace of the normal component if $\z$ is smooth, i.e.
$$
[\z,\nu^\Omega](x) = \z(x) \cdot \nu^\Omega(x) \quad \hbox{for all} \ x \in \partial \Omega  \ \ \hbox{if} \ \z \in C^1(\overline{\Omega}, \R^m).
$$
It follows from \cite[Proposition 3.1]{ChF_arma99} or \cite[Proposition 3.4]{ACM_afst05} that
{$\dive \z$ is absolutely continuous with respect to $\mathcal H^{N-1}$.
}
Therefore, given  $\z \in X_{\mathcal{M}}(\Omega)$ and $u \in BV(\Omega)\cap L^\infty(\Omega)$, the functional $(\z,Du)\in \mathcal D'(\Omega)$ given by
\begin{equation}\label{defmeasx144}
\langle (\z,Du),\psi\rangle := - \int_{\Omega} u \, \psi \d (\dive \z) - \int_{\Omega}
u \, \z \matdot \nabla \psi \d x\,
\end{equation}
is well defined, and the following  holds (see  \cite{Caselles_jde11},  Lemma 5.1, Theorem 5.3,  Lemma 5.4, \bk and Lemma 5.6).

\begin{lemma}\label{lemmacaselles}
Let $\z \in X_{\mathcal{M}}(\Omega)$ and  $u \in BV(\Omega)\cap L^\infty(\Omega)$. Then
the functional $(\z,Du)\in \mathcal D'(\Omega)$ defined by \eqref{defmeasx144} is a Radon measure which is absolutely continuous with respect to $\vert Du \vert$. Furthermore
\begin{equation}\label{Green}
\int_{\Omega} u \d(\dive \z) + (\z, Du)(\Omega) =
\int_{\partial \Omega} [\z, \nu^\Omega] u \d\mathcal{H}^{m-1}
\end{equation}
and
\begin{equation}\label{anzellotti-caselles}
  \dive(u\z )=u\dive\z+(\z,Du)\,\quad {\rm as \ measures.}
\end{equation}
\end{lemma}

%
\section{Entropy solution to the Dirichlet problem}

{
Let
}
$\tau\in (0,+\infty]$. In this Section we consider the following problem:
\begin{equation}
  \label{pbparab}
  \left\{\begin{array}
    {ll} u_t=\dys \dive\left(u^m\frac{\nabla u}{|\nabla u|}\right) & {\rm in \ }{Q_{\rm \tau}:=(0,{\rm \tau})\times\Omega}
        \\[1ex]
        u(0,x)=u_0 & {\rm in \ } \Omega
    \\
    u=g & {\rm on \ } {S_{\rm \tau}:=(0,{\rm \tau})\times\partial\Omega}
  \end{array}\right.
\end{equation}

\subsection{Definition of entropy solution}
\label{S-sol}

{A solution to problem \eqref{pbparab} is defined as follows.}

\begin{definition}\label{def-sol}
Let $u_0\in L^\infty_+(\Omega)$, $g\in L^\infty_+(\partial \Omega)$, {and ${\rm \tau}< +\infty$}. A nonnegative function  $u\in C([0,{\rm \tau});$$L^1(\Omega))\cap L^\infty((0,{\rm \tau})\times\Omega)$ is an entropy solution {to \eqref{pbparab} in $Q_\tau$}
if:
\begin{itemize}
\item[$(i)$] $\ell(u)\in L^1_{}((0,{\rm \tau}); BV(\Omega))$ for all $\ell\in\mathcal L$;
\item[$(ii)$]  $u_{t}\in L^{\infty}_{ loc,w} ((0,{\rm \tau}], \mathcal M (\Omega))$;
\item[$(iii)$] there exists $\w\in L^{\infty}((0,{\rm \tau})\times\Omega)$ such that $\|\w\|_{\infty}\leq 1$ and that $\z:=\vp (u)\w$ satisfies
{
 \begin{equation}\label{dist}
   u_t(t) =\dive \z(t) \quad\mbox{as distributions for a.e.\ $t\in (0,{\rm \tau})$;}
   \end{equation}
}
\item[$(iv)$]  the {e}ntropy inequality
\begin{equation}\label{main-ineq}
\int_{0}^{{\rm \tau}}\into \psi\ \d h (u, D\ell (u)) \leq  \int_{0}^{{\rm \tau}} \into J_{\ell}(u) \psi_{t}  - \intt\into \ell(u)\z\cdot\nabla \psi
\end{equation}
holds for any $\ell\in\mathcal L$ and any nonnegative $\psi \in C^\infty_c((0,{\rm \tau})\times\Omega)$;
\item[$(v)_{}$] for {a.e.} $t\in (0,{\rm \tau})$,  \begin{equation}\label{boundconddu>g}
\mbox{$u (t)\ge g$ \ \ $\mathcal H^{N-1}$-a.e. \ \ on $\partial \Omega$}\,,
\end{equation}
\begin{equation}\label{boundcondd=}  [\z (t),\nu^\Omega] =
-\varphi(u(t)) \ \  {\rm if \ } {u (t) > g} \quad\mbox{$\mathcal H^{N-1}$-a.e. on $\partial\Omega$;}
\end{equation}
\item[$(vi)$] $u(0)=u_0\ \ \mbox{in $L^{1}(\Omega)$}\,.$
\end{itemize}
{A nonnegative function $u$ is an entropy solution to \eqref{pbparab} in $Q:=Q_{+\infty}$ if it is an entropy solution to \eqref{pbparab} in $Q_{\rm \tau}$ for all ${\rm \tau}$.}
\end{definition}

\begin{remark}\label{here} {The normal trace of $\z$ in \eqref{boundcondd=}} makes sense since $\dive \z(t)\in \mathcal M (\Omega)$ for a.e. $0< t< {\rm \tau}$.  Moreover, as $\ell(u)\in L^1_{}([0,{\rm \tau}); BV(\Omega))$, the trace of  {$u(t)$ on $\partial \Omega$}
is well defined for a. e.  $t\in (0,{\rm \tau})$, see \cite[Lemma 5.1]{GMP}. The regularity of $u_t$ stated in $(ii)$ naturally arises from the homogeneity of the operator (see \eqref{ut_measure}; see also Remark \ref{remark3.14}). {For a discussion on the form of the Dirichlet boundary condition in $(v)$, we refer to the introduction of \cite{GMP}.}
\end{remark}

We now give a definition of subsolution to {problem \eqref{pbparab}, consistent with those previously given in literature (see e.g.\ \cite{GMP1,GMP} and references therein).}

\begin{definition}\label{def-sub}
Let $u_0\in L^\infty_+(\Omega)$, $g\in L^\infty_+(\partial \Omega)${, and ${\rm \tau}\in (0,+\infty)$. A nonnegative function} $u\in C([0,{\rm \tau});$$L^1(\Omega))\cap L^\infty((0,{\rm \tau})\times\Omega)$ is an entropy subsolution to
{\eqref{pbparab} in $Q_{\rm \tau}$ if $(i)$, $(ii)$, and $(iv)$ in Definition \ref{def-sol} hold, whilst $(iii)$, $(v)$, and $(vi)$ are replaced by:}
\begin{itemize}
{
\item[$(iii)_{sub}$] There exists $\w\in L^{\infty}((0,{\rm \tau})\times\Omega)$ such that $\|\w\|_{\infty}\leq 1$ and that $\z:=\vp (u)\w$ satisfies
 \begin{equation}\label{dist<=}
   u_t(t) {\le} \dive \z(t) \quad\mbox{as distributions in $\Omega$ for a.e.\ $t\in (0,{\rm \tau})$;}
   \end{equation}
}
\item[$(v)_{sub}$]  for {a.e.} $t\in (0,{\rm \tau})$,
\begin{equation}\label{boundcondd<=}
[\z (t),\nu^\Omega] = -\varphi(u(t)) \ \  {\rm if \ } {u (t) > g} \quad\mbox{$\mathcal H^{N-1}$-a.e. on $\partial\Omega$;}
\end{equation}
\item[$(vi)_{sub}$] $u(0)\leq u_0\ \ \mbox{in $L^{1}(\Omega)$}\,.$
\end{itemize}
{A nonnegative function $u$ is an entropy subsolution to \eqref{pbparab} in $Q:=Q_{+\infty}$ if it is an entropy subsolution to \eqref{pbparab} in $Q_{\rm \tau}$ for all ${\rm \tau}$.}

\end{definition}


\subsection{Existence}

In this subsection we will prove the following result.
\begin{theorem}\label{thm-exi}
 For any $u_0\in L^\infty_+(\Omega)$ and $g\in L_+^\infty(\partial \Omega)$ there exists an {e}ntropy solution of \eqref{pbparab} {in $Q$} in the sense of Definition \ref{def-sol}.
\end{theorem}

We consider the resolvent equation

\begin{equation}
  \label{resolvent}
  \left\{\begin{array}
    {cc} \displaystyle u-f= \dive\left(u^m\frac{\nabla u}{|\nabla u|}\right) & {\rm in \ }\Omega
    \\ u=g & {\rm on \ } \partial\Omega\,.
  \end{array}\right.
\end{equation}

In \cite[{Theorem 5.6 and 5.11}]{GMP} we obtained the following existence and uniqueness result for solutions to \eqref{resolvent}. Recall that $DTBV_+(\Omega)$ is defined in \eqref{dbtv}.

  \begin{theorem}\label{thm-elliptic}
Given $f\in L_+^\infty(\Omega)$ and $g\in L^\infty_+(\partial \Omega)$, there exists a unique {solution $u$ {to} \eqref{resolvent} in the following sense:} $u\in {DTBV_+(\Omega)}\cap L^\infty(\Omega)$, there exists $\w\in L^\infty(\Omega;\rn)$ with $\|\w\|_{\infty}\leq 1$ such that
\begin{equation}\label{resolvent_id}
u-f=\dive \z \quad\mbox{in $\mathcal D'(\Omega)$}, \quad \z:=u^m\w,
\end{equation}
\begin{equation}\label{identify-z_ell}
 \left|D \Phi({T_{a,b}^0}(u))\right| = (\z,D T_{a,b}^0(u)) \quad \mbox{as measures for a.e. } 0<a<b\leq +\infty,
\end{equation}
and
\begin{equation}\label{boundconddu>g1}
 u\geq g \ \   \mathcal H^{N-1}-{\rm a.e. \ on \ } \partial\Omega,
\end{equation}
\begin{equation}\label{boundconddu>g2}
 [\z,\nu^\Omega ]=-\varphi(u)  \  {\rm if \ } u>g \ \ \ \mathcal H^{N-1}-{\rm a.e. \ on \ } \partial\Omega.
\end{equation}
In addition, if $\tilde u\in {DTBV_+(\Omega)}\cap L^\infty(\Omega)$ is {the} solution corresponding to $\tilde f\in L^\infty_+(\Omega)$ {and $g\in L^\infty_+(\partial \Omega)$}, then
\begin{equation}
  \label{comp-elliptic}\int_\Omega { (u-\tilde u)^+} \leq \int_\Omega { (f-\tilde f)^+}.
\end{equation}

\end{theorem}

\medskip

{The solution $u$ in Theorem \ref{thm-elliptic} {satisfies} the following additional properties:}

\begin{proposition} Let $f\in L_+^\infty(\Omega)$ and $g\in L^\infty_+(\partial \Omega)$. Let $u$ be the unique solution to \eqref{resolvent} as given in Theorem \ref{thm-elliptic}. Then
\begin{equation}
\label{boundness}
0\leq u\leq M:=\max\{\|f\|_\infty,\|g\|_\infty\}
\end{equation}
and for any $\ell\in\mathcal L$, it holds:
\begin{equation}\label{identify-w}
\left|D \Phi_\ell(u)\right| = (\z,D \ell(u)) \quad \mbox{as measures};
\end{equation}
\begin{equation}
  \label{boundcondelliptic} |\Phi_\ell(g)-\Phi_\ell(u)|\leq (\ell(g)-\ell(u))[\z,\nu^\Omega]\quad \mbox{$\mathcal H^{N-1}$-a.e. on  $\partial\Omega$.}
\end{equation}
\end{proposition}

\begin{proof}
 {The bound \eqref{boundness} follows from \cite[formula (5.19)]{GMP}. For $\ell\in \mathcal L$,}
let $a>0$ be such that supp $\ell\subset [a,+\infty[$. Then $\ell(u)=\ell(T_{a,M}^0(u))$ with $M$ as given in \eqref{boundness}.
Since $u\in DTBV_+(\Omega)$, we have
\begin{eqnarray*}
(\z,D\ell (u)) &= & (\z, D\ell (T_{a,M}^0(u)))\stackrel{\text{\cite[Lemma 2.3]{GMP}}}{=}\ell'(u)(\z, DT_{a,M}^0(u))
\\ &\stackrel{\eqref{identify-z_ell}}{=}& \ell'(u) |D\Phi(T_{a,M}^0(u))| {=}|D\Phi_{\ell} (T_{a,M}^0(u))|=|D\Phi_{\ell} (u)|\,,
\end{eqnarray*}
{where in the last but one equality we used the chain rule for BV functions.}
{Inequality} \eqref{boundcondelliptic} follows directly from \eqref{boundconddu>g1} and \eqref{boundconddu>g2}; indeed, at those points  where  $u>g$ we have
$$
  [\z,\nu^\Omega](\ell (g)-\ell (u))\stackrel{\eqref{boundconddu>g2}}{=} \varphi(u)(\ell (u)-\ell (g)) = \varphi(u)\int_g^u\ell'(s)\d s  \geq  \int_g^u\ell'(s) \varphi(s)\d s\,.
$$
\end{proof}

In order to prove Theorem \ref{thm-exi}, we associate an operator in $L^1(\Omega)$ to {the following elliptic problem:
\begin{equation}
  \label{elliptic problem} \left\{\begin{array}{cc} -v=\dive\left(u^m\frac{\nabla u}{|\nabla u|}\right) & {\rm in \ }\Omega
  \\ u=g & {\rm on \ }\partial\Omega.\end{array}\right.
\end{equation}
}

{ \begin{definition} Given $g\in L_+^\infty(\partial \Omega)$, we define $B_g$ by:
$$
(u, v) \in  B_g \ \iff \
\left\{
\begin{array}{l}
u \in TBV_+(\Omega) \cap L^{\infty}(\Omega), \ v \in L^{\infty}(\Omega),
\\
{\mbox{$u$ is a solution to \eqref{elliptic problem}, }} 
\end{array}
\right.
$$
where by a solution to \eqref{elliptic problem} we mean that $u$ is a solution to \eqref{resolvent} with $f=u+v\in L^\infty_+(\Omega)$.
Accordingly, we define
$$
A_g u =  \{v\in  L^{\infty}_{ +}(\Omega) : \ (u,v) \in  B_g\}, \quad D(A_g)=\{u\in L^1_+(\Omega): \ A_g u\ne \emptyset \}.
$$
\end{definition}
}
We recall that,  on a generic Banach space $X$, an operator  $A:X\to 2^{X}$ with domain $D(A)$ is said to be accretive if
\begin{equation}
\label{def-accr}
\|u-\bar u\|_X \le \|u-\bar u+\lambda(v-\bar v)\|_X \quad\mbox{for all $\lambda >0$, $(u,v),(\bar v,\bar v)\in A$,}
\end{equation}where we use the standard identification of a multivalued operator with its graph.
Equivalently,  $A$ is accretive  in $X$ if and only if $(I+\lambda A)^{-1}$ is a single-valued non-expansive map for any $\lambda\ge 0$.

\begin{proposition}
Let $g\in L^\infty_+(\partial \Omega)$. Then {$A_g$} is an accretive operator in { $L^1(\Omega)$} with { $D(A_g)$} dense in $L^1_+(\Omega)$, satisfying the non-expansivity condition { \eqref{comp-elliptic}} and the range condition
$
L^{\infty}_+(\Omega) { \subseteq} R(I + \lambda{ A_g})
$, { for all $\lambda>0$ }
\end{proposition}
\begin{proof}
The accretivity of $A_g$ in $L^1(\Omega)$ and the range condition follow from Theorem \ref{thm-elliptic}. { Indeed, $(I+\lambda A_g)u=f$ for $\lambda > 0$ if and only if
{ $$
  \label{resolvent-l}
  \left\{\begin{array}
    {cc} \displaystyle  u -\lambda \dive\left(\varphi(u)\frac{\nabla u}{|\nabla u|}\right)=  f & {\rm in \ }\Omega.
    \\ u=g & {\rm on \ } \partial\Omega.
  \end{array}\right.
$$
Scaling $x\mapsto \hat x=\frac{1}{\lambda}x$ and applying Theorem \ref{thm-elliptic} in the rescaled domain $\hat\Omega$, we see that $I+\lambda A_g$ is single-valued and that the range condition holds true.} 
In addition,
$$
\|(u-\tilde u)^+\|_{L^1(\hat\Omega)} \stackrel{\eqref{comp-elliptic}}\le \| (f-\tilde f)^+\|_{L^1(\hat\Omega)},
$$
hence
$$
\|(u-\tilde u)^+\|_{L^1(\Omega)} \le {\| (f-\tilde f)^+\|_{L^1(\Omega)}.}
$$
Note that this implies that
$$
\|u-\tilde u\|_{L^1(\Omega)} \le \|f-\tilde f\|_{L^1(\Omega)},$$ thus $A_g$ is non-expansive.
%
%
}
To prove the density of $D({ A_g})$ in
$L^1_+(\Omega)$,
{
in view of the density of $\mathcal D_+(\Omega)$ in $L^1_+(\Omega)$, it suffices to show that any $h\in \mathcal D_+(\Omega)$ may be approximated by a sequence $\{u_n\}\subset D(A_g)$ {in $L^2(\Omega)$}.
%
%
}
By { the range condition}, $h\in R(I + \frac{1}{n} { A_g})$ for all $n \in \N$. Thus, for each $n \in \N$ there exists $u_{n} \in D({ A_g})$ such that $(u_{n}, { n(u_{n}-h})) \in B_g$.
%
%
Let $\w_n\in L^\infty(\Omega;\R^N)$ such that $\|\w\|_\infty\leq 1$ and $\z_n:=\varphi(u_n)\w_n$ as in Theorem \ref{thm-elliptic}. In particular,
$$
{ u_n-h= \frac{1}{n}\dive\z_n} \quad {\rm in  \ }\mathcal D'(\Omega).
$$
Given $\epsilon>0$, we multiply last equation by $T_{\epsilon,M} (u_n)-h$ and integrate by parts, obtaining
\begin{eqnarray*}
\int_\Omega {(u_n-h)(T_{\epsilon,M} (u_n)-h)} & { \le} & -\frac{1}{n}|D\Phi (T_{\epsilon,M}(u_n))|(\Omega) \\ & + & \frac{\varphi(M)}{n}(\|\nabla h\|_1+{ M} {\rm Per }(\Omega)).
\end{eqnarray*}
Then, letting $\epsilon\to 0^+$ we obtain that
$$
\|u_n-h\|_{L^2(\Omega)}\leq \frac{C}{\sqrt{n}}.
$$
Therefore $u_n$ has the desired property.
%
%
\end{proof}

We are now ready to begin the proof of Theorem \ref{thm-exi}.

\begin{proof}[{P{\bf roof of Theorem \ref{thm-exi}, first part}}]
Let ${\cal B}_g$ be the closure of $B_g$ in $(L^1(\Omega))^2$:
$$
\mbox{$(u, f) \in  {\cal B}_g$ $\iff$  $\exists (u_n,f_n)\in B_g: (u_n,f_n)\to (u,f)$ in $(L^{1}(\Omega))^2$.}
$$
Accordingly, we define
$$
{\cal A}_g u =  \{f\in  L^{1}_{+}(\Omega) : \ (u, f) \in  {\cal B}_g\}, \quad D({\cal A}_g)=\{u\in L^1_+(\Omega): \ {\cal A}_g u\ne \emptyset \}.
$$
It follows that ${\cal A}_g$ is accretive in $L^1(\Omega)$ (cf. \eqref{def-accr}), it satisfies the contraction principle (cf. \eqref{comp-elliptic}), and it verifies the range condition $\overline{D({\cal A}_g)}^{L^1(\Omega)} = L^1_+(\Omega) \subset
R(I + \lambda {\cal A}_g)$ for all $\lambda > 0$.
Therefore, according to Crandall-Liggett's Theorem (\cite{CL_71}, see also \cite[Theorem A.28]{ACMBook}),
for any $0 \leq u_0 \in L^1(\Omega)$ there exists a unique mild
solution (see \cite[Definition A.5]{ACMBook}) $u \in C({[0,+\infty)}; L^1(\Omega))$ of the abstract Cauchy
problem
$$
u^{\prime} (t) + {\cal { A}}_g u(t) \ni 0, \ \ \ \ u(0) = u_0.
$$
Moreover, $u(t) =  S(t) u_0$ for all $t \geq 0$, where $(\mathcal
S(t))_{t \geq 0}$ is the semigroup in $L^1(\Omega)$ generated by
Crandall-Liggett's exponential formula, i.e.,
$$
S(t) u_0 = \lim_{n \to \infty} \left(I + \frac{t}{n} {\cal { A}}_g\right)^{-n} u_0.
$$
We are going to prove that the mild solution obtained by Crandall-Ligget's Theorem is in fact an entropy solution in the sense of Definition \ref{def-sol}.

\medskip

{Fix any $\tau> 0$.} Let $k\in\N$, $h:={\rm \tau}/k$, $u^0=u_0$, and let $u^{n+1}$, $n\geq 0$, be the unique solution to the Euler implicit scheme \begin{equation}
  \label{eulerimplicit}\left\{\begin{array}
    {lc} \displaystyle \frac{u^{n+1}-u^n}{h}=\dive\left(\varphi(u^{n+1})\frac{\nabla u^{n+1}}{|\nabla u^{n+1}|}\right) & {\rm in \ } \Omega
    \\[1ex]  u^{n+1}=g & {\rm on \ }\partial \Omega,
  \end{array}\right.
\end{equation}
as given by Theorem \ref{thm-elliptic}. Note that, by \eqref{boundness},
\begin{equation}
  \label{boundparabolic} 0\leq u^n\leq M:=\max\{\|u_0\|_\infty,\|g\|_{L^\infty(\partial\Omega)}\} \quad\mbox{for all $n\ge 0$.}
\end{equation}
Let $\w^{n+1}$ be the vector field associated to $u^{n+1}$, as given by Theorem \ref{thm-elliptic}, $\z^{n+1}:=\varphi(u^{n+1})\w^{n+1}$, $t_n:=nh$, and $I_n:=(t_n,t_{n+1}]$. We define
\begin{eqnarray}
\nonumber & \displaystyle u_k := u^0\chi_{[0,t_1]}+\sum_{n=1}^{k-1} u^n\chi_{I_n}, \qquad \xi_k := \sum_{n=0}^{k-1}\frac{u^{n+1}-u^n}{h}\chi_{I_n}, &
\\  \label{def-z-k}
& \displaystyle \w_k := \w^1\chi_{[0,t_1]}+\sum_{n=1}^{k-1} \w^{k+1}\chi_{I_n}, \qquad \z_k := \varphi(u_k) \w_k. &
\end{eqnarray}
We know ({see e.g.} {\cite[Theorem {A.24 and} A.25]{ACMBook}}) that  this scheme converges{, as $k\to +\infty$,} to the unique mild solution $u(t)=S(t)u_0$ {in $(0,{\rm \tau})$}, with
\begin{equation}
\label{uk-conv-unif}
u_k\to u \quad {\rm in \ } L^1(\Omega) {\rm \ uniformly  \ in \ } [0,{{\rm \tau}]}
\end{equation}
and that, for any two given functions $u_0,\overline u_0\in L^1(\Omega)_+$, there holds
$$
\|S(t)u_0-S(t)\overline u_0\|_1\leq \|u_0-\overline u_0\|_1.
$$
Moreover, the homogeneity 
of $\mathcal B_g$ implies (cf. \cite{BenilanCrandallCA}) that there exists $C>0$ such that
$$
\overline  \lim_{h\to 0}\left\|\frac{S(t+h)u_0-S(t)u_0}{h}\right\|\leq C\frac{\|u_0\|_1}{t},
$$
which implies that
\begin{equation}
  \label{ut_measure}\|tu_t\|_{L^\infty((0,{\rm \tau});\mathcal M(\Omega))}\leq C\|u_0\|_1.
\end{equation}
Arguing as in \cite[{Proof of Theorem 1}]{ACMM-MA}, we find that
$$
\w_k\stackrel{*}\rightharpoonup \w \quad {\rm weakly }^* {\rm \ in \ } L^\infty(Q_{{\rm \tau}}),{\quad\|\w\|_\infty \le 1,}
$$
\begin{equation}\label{zk-conv}
\z_k \stackrel{*}\rightharpoonup \varphi(u)\w=:\z \quad {\rm weakly}^* {\rm \ in \ } L^\infty(Q_{{\rm \tau}}),
\end{equation}
$$
\xi_k\stackrel{*}\rightharpoonup u_t \quad {\rm weakly}^* {\rm \ in \ }(L^1_{}((0,{\rm \tau});BV(\Omega)\cap L^2(\Omega)))^*
$$
and
$$
u_t=\dive \z \quad {\rm in \ } {\mathcal D'(Q_{\rm \tau}).}
$$
In fact, by \eqref{ut_measure}, we have
$$
u_t=\dive \z \quad {\rm in \ } L^\infty_{loc{,w}}((0,{{\rm \tau}]},\mathcal M (\Omega)),
$$
{hence $(iii)$ in {Definition \ref{def-sol}} holds.} Moreover, by \cite[Lemma 10]{ACMM-MA},  it holds
\begin{equation}
  \label{trace-conv} [\z^k,\nu^\Omega]\rightharpoonup [\z,\nu^\Omega] \,,\quad {\rm weakly}^* {\rm \ in \ } L^\infty(S_{\rm \tau}).
\end{equation}
This completes the first part of the proof of Theorem \ref{thm-exi}.
\end{proof}

The proof of Theorem \ref{thm-exi} will be completed once the following three lemmas (Lemma \ref{lem:ubv}, Lemma \ref{2.9}, and  Lemma \ref{lem-boundcondparabolic}) have been established.

\begin{lemma}\label{lem:ubv}
Let {$u_0\in L^\infty_+(\Omega)$, $g\in L_+^\infty(\partial \Omega)$, $\tau\in (0,+\infty)$,}
and $u(t)=S(t)u_0$. Then $u\in L^1((0,{\rm \tau});TBV_+(\Omega))$ and
\begin{equation}\label{bvtimespace}
\ell(u),J_\ell(u)\in BV([\overline\tau,{\rm \tau}]\times\Omega)\quad \mbox{for any $\ell\in\mathcal L$ and any $\overline\tau>0$.}
\end{equation}
%
%
\end{lemma}
\begin{proof}
{Let $u^n$ be defined by \eqref{eulerimplicit}.}
We multiply the first equation in  \eqref{eulerimplicit} by $\ell(u^{n+1})$ and integrate by parts:
$$
\int_\Omega \ell(u^{n+1})\frac{u^{n+1}-u^n}{h}=\int_\Omega \ell(u^{n+1})\dive \z^{n+1} {\stackrel{\eqref{identify-w}}=} \int_{\partial\Omega} \ell(u^{n+1})[\z^{n+1},\nu^\Omega]\d \mathcal H^{N-1}-\int_\Omega |D\phi_\ell(u^{n+1})|.
$$
Then, using the  convexity of $J$, one gets
$$
\int_\Omega \frac{J_\ell(u^{n+1})-J_\ell(u^n)}{h}+\int_\Omega  |D\phi_\ell(u^{n+1})|\leq \int_{\partial\Omega} \ell(u^{n+1})[\z^{n+1},\nu^\Omega]\d \mathcal H^{N-1}.
$$
Integrating over $I_{n+1}$ and adding up, we get
\begin{eqnarray*}
\lefteqn{\int_{h{={\rm \tau}/k}} ^{\rm \tau}\int_\Omega |D\phi_\ell(u_{k})| = \sum_{n=0}^{k-1}\int_{I_{n+1}}\int_\Omega |D\phi_\ell(u^{n+1})|}
\\ &\leq & { -}  \sum_{n=0}^{k-1}\int_{I_{n+1}}\int_\Omega \frac{J_\ell(u^{n+1})-J_\ell(u^n)}{h}+\sum_{n=0}^{k-1}\int_{I_{n+1}}\int_{\partial\Omega}\ell(u^{n+1})[\z^{n+1},\nu^\Omega]\d \mathcal H^{N-1}
\\ &\stackrel{\eqref{boundparabolic},\eqref{def-z-k}}\leq & \int_\Omega J_\ell(u_0)-\int_\Omega J_\ell(u^k)+\sum_{n=0}^{k-1}\int_{I_{n+1}}\int_{ \partial\Omega} \ell(M)\varphi(M)\,.
%
%
\end{eqnarray*}
By lower semicontinuity {and \eqref{uk-conv-unif}}, we get that
$$
\int_\Omega J_\ell(u) + \int_{0}^{\rm \tau}\int_\Omega |D\phi_\ell({ u})|\leq {\rm \tau} |\partial\Omega| \ell(M)\varphi(M){+\int_\Omega J_\ell(u_0)\,.}
$$
Hence $u\in L^1((0,{\rm \tau});TBV_+(\Omega))$ and \eqref{bvtimespace} follows taking \eqref{ut_measure} into account.
\end{proof}

The next result is preparatory for the proof of $(iv)$ and $(v)$.


\begin{lemma}\label{2.9}
  The following inequality is satisfied for any $0\leq\psi\in C_c^\infty((0,{\rm \tau})\times\overline \Omega)$ and any $\ell\in\mathcal L$:
 \begin{equation}\label{entr-con-with-bdr}\int_0^{\rm \tau}\int_\Omega \psi d|D\phi_\ell(u)|+\int_0^{\rm \tau}\int_{\partial\Omega}\psi|\phi_\ell(u)-\phi_\ell(g)|$$$$\leq \int_0^{\rm \tau}\int_\Omega J_\ell(u)\psi_t+\int_0^{\rm \tau}\int_{\partial\Omega} \ell(g)[\z,\nu^\Omega]\psi \d\mathcal H^{N-1}-\int_0^{\rm \tau}\int_\Omega \ell(u)\z\cdot \nabla \psi.\end{equation}
\end{lemma}
\begin{proof}
As in the proof of Lemma \ref{lem:ubv}, we multiply the equation by $\ell(u^{n+1})\psi$ and integrate by parts to get
\begin{eqnarray*}
\lefteqn {\int_\Omega \frac{J_\ell(u^{n+1})-J_\ell(u^n)}{h}\psi+\int_\Omega \psi |D\phi_\ell(u^{n+1})|}
\\ & \leq & \int_{\partial\Omega} \ell(u^{n+1})\psi [\z^{n+1},\nu^\Omega]\d\mathcal H^{N-1}-\int_\Omega \ell(u^{n+1})\z^{n+1}\cdot\nabla \psi.
\end{eqnarray*}
Integrating over $I_{n+1}$, adding up, and choosing $k$ sufficiently large such that ${\rm supp} \psi\subset {(h,{\rm \tau}-h)\times\overline \Omega}$, we see that
\begin{eqnarray*}
\lefteqn{\int_0^{\rm \tau} \int_\Omega J_\ell(u_k)\frac{\psi(t)-\psi(t+h)}{h}+\int_0^{\rm \tau}\int_\Omega \psi|D\phi_\ell(u_k)|}
\\ & \leq & \int_0^{\rm \tau}\int_{\partial\Omega}\ell(u_k(t))[\z_k(t-h),\nu^\Omega]\psi(t)\d\mathcal H^{N-1}-\int_0^{\rm \tau}\int_\Omega \ell(u_k(t))\z_k(t-h)\cdot\nabla\psi.
\end{eqnarray*}
Using \eqref{boundcondelliptic}, we obtain that
\begin{eqnarray*}
\lefteqn{\int_0^{\rm \tau}\int_\Omega \psi|D\phi_\ell(u_k)|+\int_0^{\rm \tau}\int_{\partial\Omega}|\phi_\ell(u_k)-\phi_\ell(g)|\psi\d\mathcal H^{N-1}} \\ & \leq & \dys \int_0^{\rm \tau} \int_\Omega J_\ell(u_k)\frac{\psi(t+h)-\psi(t)}{h}+\int_0^{\rm \tau}\int_{\partial\Omega}\ell(g)[\z_k(t-h),\nu^\Omega]\psi(t)\d\mathcal H^{N-1}-\int_0^{\rm \tau}\int_\Omega \ell(u_k(t))\z_k(t-h)\cdot\nabla\psi.\end{eqnarray*}
We pass to the limit as $k\to +\infty$:  by lower semicontinuity, {\eqref{uk-conv-unif}, \eqref{zk-conv}, and } \eqref{trace-conv} we obtain \eqref{entr-con-with-bdr}.
\end{proof}

We next show that the solution also satisfies inequality \eqref{boundcondelliptic} a.e. in $[0,{\rm \tau}]$:
\begin{lemma}
  \label{lem-boundcondparabolic} Let $u\in C((0,{\rm \tau});L^1(\Omega))\cap L^\infty(Q_{\rm \tau})\cap L^1((0,{\rm \tau});TBV(\Omega))$ and $\w\in X(\Omega)$ with $\|\w\|_\infty\leq 1$ such that $u$ and $\z:=\varphi(u) \w$ satisfy the entropy inequality \eqref{entr-con-with-bdr}. { Then,}
 \begin{equation}
  \label{boundcondparabolic} |\Phi_\ell(g)-\Phi_\ell(u(t))|\leq (\ell(g)-\ell(u(t)))[\z(t),\nu^\Omega]\quad \mbox{$\mathcal H^{N-1}-$a.e. on  $\partial\Omega$ for a.e.\ $t>0$.}
\end{equation}
\end{lemma}

\begin{proof}
  It suffices to integrate by parts equation \eqref{entr-con-with-bdr} (recall \eqref{bvtimespace}) to get
  \begin{eqnarray*}
\lefteqn{\int_0^{\rm \tau}\int_\Omega \psi d|D\phi_\ell(u)|+\int_0^{\rm \tau}\int_{\partial\Omega}\psi|\phi_\ell(u)-\phi_\ell(g)|}\\ & \leq & \dys-\int_0^{\rm \tau}\int_\Omega (J_\ell(u))_t\psi+\int_0^{\rm \tau}\int_{\partial\Omega} (\ell(g)-\ell(u))[\z,\nu^\Omega]\psi \d\mathcal H^{N-1}+\int_0^{\rm \tau}\int_\Omega \psi (\z,D\ell(u)).\end{eqnarray*}This implies that, a.e. $t\in [0,{\rm \tau}]$ as measures $$|D\phi_\ell(u)|+|\phi_\ell(u)-\phi_\ell(g)|\mathcal H^{N-1}\res_{\partial\Omega}\leq -(J_\ell(u))_t+(\ell(g)-\ell(u))[\z,\nu^\Omega]\mathcal  H^{N-1}\res_{\partial\Omega}+(\z,D\ell(u)).$$Since they have disjoint support, we obtain, a.e. $t\in [0,{\rm \tau}]$ as measures, $$|D\phi_\ell(u)|\leq -(J_\ell(u))_t+(\z,D\ell(u)),$$$$|\phi_\ell(u)-\phi_\ell(g)|\mathcal H^{N-1}\res_{\partial\Omega}\leq (\ell(g)-\ell(u))[\z,\nu^\Omega]\mathcal  H^{N-1}\res_{\partial\Omega},$$which proves the Lemma.
\end{proof}


We are now ready to complete the proof of Theorem \ref{thm-exi}.

\begin{proof}[P{\bf roof of Theorem \ref{thm-exi}: conclusion}]
Let $\tau\in (0,+\infty)$. In the first part of the proof, we have already shown that $u\in C([0,{\rm \tau});$$L^1(\Omega))\cap L^\infty((0,{\rm \tau})\times\Omega)$ and that $(ii)$, $(iii)$, and $(vi)$ {in Definition \ref{def-sol}} hold. Lemma \ref{lem:ubv} implies $(i)$. Lemma \ref{2.9} with $\psi\in C_c^\infty((0,{\rm \tau})\times\Omega)$ implies $(iv)$.  { We note that  $(v)$ is implied by \eqref{boundcondparabolic}  as proven in \cite[Lemma 5.8]{GMP}.}
\end{proof}

\subsection{Uniqueness}

{In this section we prove:}

\begin{theorem}\label{thm-uni}
Let $u_0\in L^\infty_+(\Omega)$ and $g\in L^\infty_+(\partial \Omega)$. The {entropy} solution to \eqref{pbparab} {in $Q$} is unique.
\end{theorem}


The proof of Theorem \ref{thm-uni} is a consequence of  the following comparison result:

\begin{theorem}\label{T-sub}
{Let ${\rm \tau}>0$, $u_0\in L^\infty_+(\Omega)$, and $g\in L^\infty_+(\partial \Omega)$. Let $u$, resp. $\underline u$, be an entropy solution, resp. subsolution, to \eqref{pbparab} in $Q_{\rm \tau}$.}
Then $\underline u(t)\le u(t)$ for all $t\in (0,{\rm \tau})$.
\end{theorem}

\begin{proof}[P{\bf roof of Theorem \ref{T-sub}}]

 The basic idea in the proof of Theorem \ref{T-sub} relies in a refinement of the proofs of   \cite[Theorem 2.6]{GMP1} and of \cite[Theorem 3]{ACMM-MA} (with the emendations given in \cite{ACMM-MA-E}). We divide the proof into steps.

\medskip

\noindent $\bullet$ {\it Step 0.  Preparatory tools}.

\medskip
For $S,T\in \mathcal{T}_+$  and  ${u}$ satisfying (i) in Definition \ref{def-sol},    we let  $h_S(u,DT(u))$ be the Radon measure defined for a.e. $t \in [0,{\rm \tau}]$ by
\begin{eqnarray}
\nonumber \langle h_S(u,DT(u)), \phi\rangle  &:=& \int_{\Omega} \phi S(T^0(u)) h(T^0(u),\tilde\nabla T^0(u)) \\ && + \int_{\Omega} \phi  \d |D^j J_{S\varphi}(T^0(u))| \, + \int_{\Omega} \phi S(T^0(u)) h(T^0(u),\tilde\nabla T^0(u))\nonumber  \\ && + \int_{J(T^0(u))} \phi \int_{T^0(u)^-}^{T^0(u)^+}S(s) \varphi(s)  \d s  \d \mathcal{H}^{N-1}\quad\mbox{for all $\phi\in C_c(\Omega)$}
 \label{lsc}
\end{eqnarray}
For $b > a > 2\eps > 0$, we let $T(r) =T_{a,b}^{a}(r).$ Without losing generality (\cite[Lemma 1] {ACMM-MA-E}) we can choose $\eps$ such that
\begin{equation}\label{preparatory 1}
\mathcal{L}^{N+2}(\{(x,s,t): T_{a,\infty}^0(\underline{u}(s,x) ) -  T^{0}_{\frac{a}{2},\infty}({u} (t,x) ) =\eps\})=0\,,
\end{equation}
and
\begin{equation}\label{preparatory 2}
\int_{(0,{\rm \tau})^2} (|D^c T^{0}_{a,\infty}(u(t) )|+|D^c T^{0}_{a,\infty}(u(t) )|)(\{{{ T^{0}_{a,\infty}(u(s) )}}+ {{T^{0}_{\frac{a}{2},\infty}(u(t) )}}=\eps\})dsdt =0\,.
\end{equation}

\noindent $\bullet$ {\it Step 1. Doubling.}

\medskip

We denote $\z=\vp (u) \w$ and $\underline \z=\vp (\underline{u})\underline \w$.  {We} define
\begin{eqnarray}\label{def-R}
R_{\eps, l}(r) &:=& \left\{\begin{array}{ll} T_{l-\eps,l}^{l-\eps}(r) & \mbox{if $l>2\eps$},
\\
T_{\eps,2\eps}^\eps(r)  & \mbox{if $l<2\eps$},
\end{array}\right.
\\ \label{def-S}
S_{\eps,l}(r) &:=& \left\{\begin{array}{ll} T_{l,l+\eps}^l(r) & \mbox{if $l>\eps$},
\\
T_{\eps,2\eps}^\eps(r) & \mbox{if $l<\eps$}.
\end{array}\right.
\end{eqnarray}
We choose two different pairs of variables $(t,x)\in Q_{\rm \tau} = (0,{\rm \tau})\times\Omega$, $(\ut, \ux)\in \uQ_{\rm \tau}:=(0,{\rm \tau})\times\Omega$, and consider $u$, $\z$ and $\und u$, $\und \z$ as functions of $(t,x)$, resp. $(\ut,\ux)$. Let $0 \leq \phi \in {\mathcal D}((0, {\rm \tau}))$, $0\leq \sigma\in \mathcal D(\Omega)$, $\rho_{k}$ a sequence of mollifiers in $\R^N$, and $\tilde{\rho}_n$ a sequence of mollifiers in $\R$. Define
$$
\eta_{k,n}(t, x, \ut, \ux):= \rho_k(x - \ux) \tilde{\rho}_n(t - \ut) \phi
\bigg(\frac{t+\ut}{2} \bigg){ \sigma
\bigg(\frac{x+\ux}{2}} \bigg).
$$
For fixed $(\ut,\ux)$, we choose $\ell(u)=\elleuu(u)=T(u)R_{\eps,\underline{u}}(u)$ and $\psi=\eta_{k,n}$ in \eqref{main-ineq}:
\begin{equation}\label{2UE5}
\begin{array}{l}
\dys -\int_{Q_{\rm \tau}} {J}_{\elleuu}(u)(\eta_{k,n})_t
+ \int_{Q_{\rm \tau}} \eta_{k,n} \d h(u, D_x (TR_{\eps,\uu}(u))) + 
  \int_{Q_{\rm \tau}} T(u) R_{\eps,
\underline{u}}(u) \z \cdot \nabla_x \eta_{k,n} \leq 0.
\end{array}
\end{equation}
Similarly, for fixed $(t,x)$ we choose $\ell(\uu)=\elleu(\uu)=T(\uu)S_{\eps,u}(\uu)$ and $\psi=\eta_{k,n}$ in \eqref{main-ineq} (which holds for the subsolution $\underline u$):
\begin{equation}\label{2UE6}
\begin{array}{l}
\dys- \int_{\underline Q_{\rm \tau}} {J}_{\elleu}(\uu)(\eta_{k,n})_t
+ \int_{\underline Q_{\rm \tau}} \eta_{k,n}\d h(\uu, D_{\ux} (TS_{\eps,u}(\uu)))  
+  \int_{\underline Q_{\rm \tau}}{T(\uu)} S_{\eps,u}(\uu) \z
\cdot \nabla_{\underline x} \eta_{k,n} \leq 0.
\end{array}
\end{equation}
Integrating (\ref{2UE5}) in $\uQ_{\rm \tau}$, (\ref{2UE6}) in $Q_{\rm \tau}$, adding the two inequalities and taking into account that $\nabla_x \eta_{k,n} + \nabla_\ux  \eta_{k,n} = \rho_k(x-\ux){\tilde\rho_n}(t-\ut)\phi
\left(\frac{t+\ut}{2} \right)\nabla \sigma\left(\frac{x+\ux}{2}\right)$, 
we see that
\begin{eqnarray}\nonumber
 && -\int_{Q_{\rm \tau} \times \uQ_{\rm \tau}} \left({J}_{T R_{\eps,\underline{u}}}(u) (\eta_{k,n})_t + {J}_{T S_{\eps,u}}( \underline{u}) (\eta_{k,n})_\ut \right)
+  \int_{Q_{\rm \tau}\times \uQ_{\rm \tau}}\eta_{m,n} \d h(u, D_x (TR_{\eps,\uu}(u))) \\ \nonumber && + {\int_{Q_{\rm \tau} \times \uQ_{\rm \tau}}} \eta_{k,n}\d h(\uu, D_{\ux} (TS_{\eps,u}(\uu)))
 - \int_{Q_{\rm \tau} \times \uQ_{\rm \tau}} T(u) R_{\eps, \underline{u}}(u) \z \cdot \nabla_\ux  \eta_{k,n}- \int_{Q_{\rm \tau} \times \uQ_{\rm \tau}} T(\underline{u}) S_{\eps,u}(\underline{u}) \underline{\z} \cdot  \nabla_x \eta_{k,n}
\\ \nonumber
&+& \int_{Q_{\rm \tau}
\times \uQ_{\rm \tau}} \rho_k\tilde\rho_n\phi\big(T(u) R_{\eps, \underline{u}}(u) \z+T(\uu)S_{\eps,u}(\underline{u})\underline  \z\big)\cdot\nabla \sigma\leq 0
\end{eqnarray}
{That} is, after one integration by parts,
\begin{equation}\label{2UE8}
\tilde I_1 +\tilde I_2\le 0,
\end{equation}
where
\begin{eqnarray}\nonumber
\tilde I_{1} &:=&  -\int_{Q_{\rm \tau} \times \uQ_{\rm \tau}}
\left({J}_{TR_{\eps,\underline{u}}}(u) (\eta_{k,n})_t +
{J}_{TS_{\eps,u}}( \underline{u}) (\eta_{k,n})_\ut \right)
\\ \nonumber
{\tilde I}_2 &:= & \int_{Q_{\rm \tau}\times \uQ_{\rm \tau}}\eta_{k,n} \d h(u, D_x (TR_{\eps,\uu}(u))) + {\int_{Q_{\rm \tau} \times \uQ_{\rm \tau}}} \eta_{k,n} \d h({\und u}, D_{\ux} (TS_{\eps,u}(\uu)))
\\  \nonumber &&
+ \int_{Q_{\rm \tau}
\times \uQ_{\rm \tau}}
\eta_{k,n} T(u)  \z \cdot {\d} D_{\ux}R_{\eps, \underline{u}}(u)+ \int_{Q_{\rm \tau}
\times \uQ_{\rm \tau}}
 \eta_{k,n}  T(\underline{u})\underline{\z} \cdot {\d} D_{x }S_{\eps,u}(\underline{u})
\\ \nonumber &+& \int_{Q_{\rm \tau}
\times \uQ_{\rm \tau}} \rho_m\tilde\rho_n\phi \big(T(u) R_{\eps, \underline{u}}(u) \z+T(\uu)S_{\eps,u}(\underline{u})\underline  \z\big)\cdot\nabla \sigma
\end{eqnarray}

{
By definition, $T(u)=0$ if $\{u\le a\}$ and $T(\und u)=0$ if $\{u\le a\}$. On the other hand, we have
$$
R_{\eps,l}(r)= \left\{\begin{array}{ll} T_{l-\eps,l}^{l-\eps}(r)& \mbox{if $l>a$}
\\
T_{l-\eps,l}^{l-\eps}(r)=\eps & \mbox{if $2\eps<l<a$}
\\
T_{\eps,2\eps}^\eps(r) =\eps & \mbox{if $l<2\eps$}
\end{array}\right\} = T_{l-\eps,l}^{l-\eps}(r) \quad\mbox{for $r\ge a$}
$$
and, analogously, $S_{\eps,l}(\und u)=T_{l,l+\eps}^l(\und u)$ for $\und u>a$. Therefore in $\tilde I_2$ we have
\begin{eqnarray}\label{recover1}
R_{\eps,\und u}(u) &=& T_{\und u-\eps,\und u}^{\und u-\eps}(u) = T_{0,\eps}^0(u-\und u+\eps),
\\ \label{recover2} S_{\eps,u}(\und u) &=& T_{u,u+\eps}^{u}(\und u)= T_{0,\eps}^0(\und u-u).
\end{eqnarray}
}
The latter equalities in \eqref{recover1}-\eqref{recover2} show in particular that
\begin{equation}
R_{\eps,\und u}(u)+S_{\eps,u}(\und u)\equiv \eps,
\label{change}
\end{equation}
whence
$$
D_x R_{\eps,\und u}(u)= -D_x S_{\eps,u}(\und u) \quad\mbox{and}\quad D_\ux S_{\eps,u}(\und u) =-D_\ux R_{\eps, \und u} (u).
$$
Furthermore, letting
\begin{equation}\label{def-ue}
u_\eps := T_{\und u-\eps,\und u}^0(u), \quad  \und u_\eps := T_{u,u+\eps}^0(\und u),
\end{equation}
it follows from \eqref{recover1}-\eqref{recover2} that
\begin{equation}\label{prop-ue}
D_x R_{\eps,\und u}(u)= D_x u_\eps \quad\mbox{and}\quad D_\ux S_{\eps,u}(\und u) =D_\ux \und u_\eps.
\end{equation}
{
Hence $\tilde I_2$ may be rewritten as follows (we also permute terms for future convenience):
\begin{eqnarray}\nonumber
{\tilde I}_2 &:= &
\int_{Q_{\rm \tau}\times \uQ_{\rm \tau}}\eta_{k,n} \d h(u, D_x (TR_{\eps,\uu}(u)))
- \int_{Q_{\rm \tau}
\times \uQ_{\rm \tau}}
 \eta_{k,n}  T(\underline{u})\underline{\z} \cdot \d D_{x} u_\eps
+ \int_{Q_{\rm \tau} \times \uQ_{\rm \tau}} \eta_{k,n} \d h(\und u, D_{\ux} (TS_{\eps,u}(\uu))) \\  \nonumber &&- \int_{Q_{\rm \tau}
\times \uQ_{\rm \tau}}
\eta_{k,n} T(u)  \z \cdot \d D_{\ux}\und u_\eps
+\int_{Q_{\rm \tau}
\times \uQ_{\rm \tau}} \rho_k\tilde\rho_n\phi \big(T(u) R_{\eps, \underline{u}}(u) \z+T(\uu)S_{\eps,u}(\underline{u})\underline  \z\big)\cdot\nabla \sigma
\end{eqnarray}
}

\noindent $\bullet$ {\it Step 2. A preliminary estimate on $\tilde I_2$}

\medskip

We estimate the first two terms in $\tilde I_2$. We analyze the first one (the second one is analogous). We split $h(u, D_x (TR_{\eps,\uu}(u)))$  into its diffuse and singular parts.
Using \eqref{def-ue}, \eqref{prop-ue}, and recalling  \eqref{lsc}, we have that
\begin{equation}\label{htd}
\begin{array}{l}
\dys h^{d} (u, D_x (TR_{\eps,\uu}(u)))=\varphi(u) (TR_{\eps,\uu}(u))' |\tilde{\nabla} u|\stackrel{T'\geq 0}\geq\varphi(u) TR_{\eps,\uu}'(u) |\tilde{\nabla} u|   =
\dys T (u)|\tilde{\nabla}\Phi_{R_{\eps,\uu}} u|= h^{d}_{T} (u, D_x u_\eps )\,.
\end{array}
\end{equation}
and
\begin{eqnarray}\label{hts}
\nonumber
\dys h^{j} (u, D_x (TR_{\eps,\uu}(u))) & =& |\Phi_{TR_{\eps,\uu}}(u^+) - \Phi_{TR_{\eps,\uu}}(u^-)|=\int_{u^-}^{u^+} \varphi(s)(TR_{\eps,\uu})' (s) \d s\\ &  \dys \stackrel{T'\geq 0}\geq & \int_{u^-}^{u^+} \varphi(s) T(s) R_{\eps,\uu} ' (s) \d s= h^{j}_{T} (u, D_x u_\eps)\,.
\end{eqnarray}
Therefore, using \eqref{htd} and \eqref{hts}, $\tilde I_2$ may be estimated  by
\begin{eqnarray}
\nonumber
\tilde I_2 &\geq & \int_{Q_{\rm \tau}\times \uQ_{\rm \tau}} \eta_{k,n} \d h_T(u, D_x u_\eps)- \int_{Q_{\rm \tau}\times \uQ_{\rm \tau}}
\eta_{k,n} T(\underline{u}) \ \underline{\z} \cdot \d D_x u_\eps
+ \int_{Q_{\rm \tau}\times \uQ_{\rm \tau}} \eta_{k,n} \d h_T(\underline{u}, D_\ux \underline{u}_\eps) \\ \nonumber
&& - \int_{Q_{\rm \tau}\times \uQ_{\rm \tau}}
\eta_{k,n}  T(u) \ \z \cdot \d D_\ux \und u_\eps
+ \eps  \int_{Q_{\rm \tau}\times \uQ_{\rm \tau}} \rho_{k}\tilde\rho_n\phi T(u)\z\cdot\nabla\sigma \\ \nonumber
&& + \int_{Q_{\rm \tau}\times \uQ_{\rm \tau}} \rho_{k}\tilde\rho_n\phi S_{\eps,u}(\uu)(T(\uu)\underline\z-T(u)\z)\cdot\nabla\sigma
:=  I_{2}+ I_{\sigma}\,, \label{def-I2}
\end{eqnarray}
where in the last step we added and subtracted $S_{\eps,u}(\uu)T(u) \z$ and we used \eqref{change}, and we defined
\begin{equation}
\dys I_{\sigma} = \eps  \int_{Q_{\rm \tau}\times \uQ_{\rm \tau}} \rho_{k}\tilde\rho_n\phi T(u)\z\cdot\nabla\sigma + \int_{Q_{\rm \tau}\times \uQ_{\rm \tau}} \rho_{k}\tilde\rho_n\phi S_{\eps,u}(\uu)(T(\uu)\underline\z-T(u)\z)\cdot\nabla\sigma.
\end{equation}

We will now split, and analyze separately, $I_{2}$ into $I_2= I_{2}^d + I_{2}^j$, where $I_{2}^d$ and $I_{2}^j$ contain the diffuse, resp. the jump, part of the measures within $I_{2}$. We note for further {reference} that, in view of \eqref{def-ue} and \eqref{recover1}-\eqref{recover2}, we have
\begin{equation}\label{def-chie}
\mbox{$\nabla_x u_\eps=\chi_\eps \nabla_x u$ and  $\nabla_\ux \und u_\eps =\chi_\eps \nabla_\ux \und u$,\quad where }\ \chi_\eps:= \chi_{\{u<\underline u<u+\eps\}} = \chi_{\{\underline u-\eps <u <\und u\}}.
\end{equation}

{\noindent $\bullet$ {\it Step 3. Estimate of the diffuse part of $ I_2$. }}

\medskip

Let us estimate the first two integrals of $I_2^d$ (see \eqref{def-I2}) uniformly with respect to $k$. 
\begin{eqnarray}
  &&\dys   \int_{Q_{\rm \tau}\times \uQ_{\rm \tau}} \eta_{k,n} \d h^d_T(u, D_x u_\eps) - \dys  \int_{Q_{\rm \tau}\times \uQ_{\rm \tau}}\nonumber
\eta_{k,n} T(\underline{u}) \ \underline{\z} \cdot \d \tilde{\nabla}_x u_\eps  \\  &&=\nonumber
\dys  \int_{Q_{\rm \tau}\times \uQ_{\rm \tau}}  \eta_{k,n} (T(u_\eps ) \varphi (u_\eps ) |\tilde{\nabla}_x u_\eps  | \dys  -T(\uu_\eps ) \varphi (\uu_\eps ) \underline \w{\cdot}\tilde{\nabla}_x u_\eps  )\\ &&
\dys = \int_{Q_{\rm \tau}\times \uQ_{\rm \tau}}  \eta_{k,n} (T(u_\eps ) \varphi (u_\eps ) \dys  -T(\uu_\eps ) \varphi (\uu_\eps ) ) |\tilde{\nabla}_x u_\eps  | +\dys \int_{Q_{\rm \tau}\times \nonumber\uQ_{\rm \tau}}  \eta_{k,n} T(\uu_\eps ) \varphi (\uu_\eps )  ( |\tilde{\nabla}_x u_\eps  | - \w\tilde{\nabla}_x u_\eps  ) \\ &&
\dys \stackrel{\|\w\|_{\infty}\leq 1}\geq \int_{Q_{\rm \tau}\times \uQ_{\rm \tau}}  \eta_{k,n} (T(u_\eps ) \varphi (u_\eps ) \dys  -T(\uu_\eps ) \varphi (\uu_\eps ) ) |\tilde{\nabla}_x u_\eps  |,  \label{hd1}
\end{eqnarray}
{
where we added and subtracted $T(\uu_\eps ) \varphi (\uu_\eps ) |\tilde{\nabla}_x u_\eps|$. Note that the {above} expression makes sense since
$$
\supp(T(\uu_\eps ) |\tilde{\nabla}_x u_\eps|)\subseteq \{\und u -\eps \le u\le \und u \ \wedge \ \und u_\eps\ge a \} \subseteq \{u\ge a-\eps \ \wedge \ \und u \ge a \}.
$$
}
Analogously we can estimate the  third and the fourth integrals {in}  $I_2^d$, to get
\begin{eqnarray}
  &&\dys  \int_{Q_{\rm \tau}\times \uQ_{\rm \tau}} \eta_{k,n} \d h_{T}^{d}(\underline{u}, D_\ux \underline{u}_\eps)  - \nonumber\int_{Q_{\rm \tau}\times \uQ_{\rm \tau}}
\eta_{k,n}  T(u) \ \z \cdot \d \tilde\nabla{_\ux} \und u_\eps  \\  &&\geq \int_{Q_{\rm \tau}\times \uQ_{\rm \tau}}  \eta_{k,n} (T(\uu_\eps ) \varphi (\uu_\eps )  - T(u_\eps ) \varphi (u_\eps )) \dys  |\tilde\nabla{_\ux} \und u_\eps  |,\label{hd2}
\end{eqnarray}
{where in this case
$$
\supp(T(u_\eps ) |\tilde{\nabla}_\ux \uu_\eps|)\subseteq \{u \le \und u\le u+\eps \ \wedge \ u_\eps\ge a \} \subseteq \{u\ge a \ \wedge \ \und u \ge a \}.
$$
}
Adding \eqref{hd1} and \eqref{hd2}{, and recalling \eqref{def-chie}}, we then get
\begin{eqnarray}
I_2^d &\geq& \int_{Q_{\rm \tau}\times \uQ_{\rm \tau}}  \eta_{k,n} (T(u_\eps ) \varphi (u_\eps ) \dys  -T(\uu_\eps ) \varphi (\uu_\eps ) ) (|\tilde{\nabla}_x u_\eps  | -|\tilde\nabla{_\ux} \und u_\eps  | )\nonumber
\\ &=& \int_{Q_{\rm \tau}\times \uQ_{\rm \tau}}  \eta_{k,n}\chi_\varepsilon (T(u ) \varphi (u ) \dys  -T(\uu ) \varphi (\uu) ) (|\tilde{\nabla}_x u  | -|\tilde\nabla{_\ux} \und u  | ){
=I_2^{ac}+I_2^{c}.}
\nonumber
\end{eqnarray}
{ with
$$
I_2^{ac}=\int_{Q_{\rm T}\times \uQ_{\rm T}}  \eta_{k,n}\chi_\varepsilon (T(u ) \varphi (u ) \dys  -T(\uu ) \varphi (\uu) ) (|{\nabla}_x u  | -|\nabla{_\ux} \und u  | )
$$
and
$$
I_2^{c}=\int_{Q_{\rm T}\times \uQ_{\rm T}}  \eta_{k,n}\chi_\varepsilon (T(u ) \varphi (u ) \dys  -T(\uu ) \varphi (\uu) ) ({\d}|D^c_x u  | -{\d}|D{_\ux}^c \und u  | )
$$
Concerning the absolutely continuous part,} since the map $s\mapsto T(s)\varphi(s)$ increasing and $u< \und u$,
\begin{eqnarray*}
I_2^{ ac} &{\geq } &  -\int_{Q_{\rm \tau}\times \uQ_{\rm \tau}}  \eta_{k,n}\chi_\varepsilon (T(\uu ) \varphi (\uu) - T(u ) \varphi (u )){|{\nabla}_x u  -\nabla{_\ux} \und u  |}
\\
 &  = &  -\int_{Q_{\rm \tau}\times \uQ_{\rm \tau}}  \eta_{k,n}\chi_\varepsilon \chi_{\{u\geq \frac{a}{2}\}}\chi_{\{\uu \geq a\}} (T(\uu ) \varphi (\uu) - T(u ) \varphi (u )){|{\nabla}_x u  -\nabla{_\ux} \und u  |}.
\end{eqnarray*}
Let
\begin{equation}\label{def-chie-t}
\hat\chi_\eps:= \chi_{\{ T_{\frac{a}{2},\infty}^{0}(u)<T_{a,\infty}^0(\uu )< T_{\frac{a}{2},\infty}^{0}(u)+\eps\}} = \chi_{\{T_{a,\infty}^0(\uu )-\eps < T_{\frac{a}{2},\infty}^{0}(u) <T_{a,\infty}^0(\uu )\}}.
\end{equation}
Since the map $s\mapsto T(s)\varphi(s)$ il locally Lipschitz in $[0,+\infty)$ and $u,\uu$ are bounded,
\begin{eqnarray*}
I_2^{ ac} & \ge &  -C\int_{Q_{\rm \tau}\times \uQ_{\rm \tau}}  \eta_{k,n}\chi_\varepsilon \chi_{\{u\geq \frac{a}{2}\}}\chi_{\{\uu \geq a\}} (\uu - u)|{\nabla}_x u  -\nabla{_\ux} \und u  |
\\  &=& -C\int_{Q_{\rm \tau}\times \uQ_{\rm \tau}}  \eta_{k,n}\hat\chi_\varepsilon \chi_{\{u\geq \frac{a}{2}\}}\chi_{\{\uu \geq a\}} (T_{a,\infty}^0(\uu ) - T_{\frac{a}{2},\infty}^{0}(u))|{\nabla}_x T_{\frac{a}{2},\infty}^{0}(u)  -\nabla{_\ux} T_{a,\infty}^0(\uu )  |
\\  &\ge & -C\int_{Q_{\rm \tau}\times \uQ_{\rm \tau}}  \eta_{k,n}\hat\chi_\varepsilon (T_{a,\infty}^0(\uu ) - T_{\frac{a}{2},\infty}^{0}(u))|{\nabla}_x T_{\frac{a}{2},\infty}^{0}(u)  -\nabla{_\ux} T_{a,\infty}^0(\uu )  |.
\end{eqnarray*}
Recalling \eqref{preparatory 1}  and using \cite[ Lemma 5]{ACMM-MA-E}, \color{black}  we obtain
\begin{eqnarray}
\label{auxiliary_for_ac_part}
 \liminf_{k\to\infty}I_2^{ ac}
&\geq&
-C \int_{(0,\tau)\times \uQ_{\rm \tau}}  \tilde \rho_n(t-\underline{t})\phi(\tfrac{t+\underline t}{2})\sigma(\ux)\hat\chi_\varepsilon(T_{a,\infty}^0(\uu )-T_{\frac{a}{2},\infty}^{0}(u)) |{\nabla}_{\und x} (T_{\frac{a}{2},\infty}^{0}(u)   - T_{a,\infty}^0(\und u))  |  \nonumber
\\  &\stackrel{\eqref{def-chie-t}}\geq & -C\varepsilon \int_{(0,\tau)\times \uQ_{\rm \tau}}  \tilde \rho_n(t-\underline{t})\phi(\tfrac{t+\underline t}{2})\sigma(\ux)\hat\chi_\eps |{\nabla}_{\und x} (T_{\frac{a}{2},\infty}^{0}(u)  - T_{a,\infty}^0(\und u))  |
\\  &\geq&\nonumber  -C\varepsilon o_\varepsilon(1),
\end{eqnarray}
where in this formula $u=u(t,\und x)$ and where in the last inequality we used the coarea formula. 

%

{ We now estimate $ I_2^c$. We note that
\begin{eqnarray*} \dys  I_2^c & = & \int_{Q_\tau\times \underline {Q_\tau}} \eta_{k,n}\chi_\varepsilon\chi_{\{u>{ \frac{a}{2}}\}}\chi_{\{\uu>a\}}(T(u)\varphi(u)-T(\uu)\varphi(\uu))(d|D^c_x u|-d|D^c_{\ux} \uu|).\end{eqnarray*}
{In view of \eqref{preparatory 2},} we may use \cite[Lemma 4]{ACMM-MA-E}, with $F(r)=T(r)\varphi(r)${,} $\omega=T_{a,\infty}^0(\uu)$ and $\underline \omega=T_{{\frac{a}{2}},\infty}^0(u)$, to get
$$\liminf_{k\to\infty}I_2^c = \int_{(0,\tau)\times\underline{Q_\tau}}\tilde \rho_n(t-\ut)\phi\left(\frac{t+\ut}{2}\right)\sigma(\ux)\chi_\varepsilon\chi_{\{u>{ \frac{a}{2}}\}}\chi_{\{\uu>a\}}(T(u)\varphi(u)-T(\uu)\varphi(\uu))(d|D^c_{\ux} u|-|D^c_{\ux} \uu|),$$where in this formula $u=u(t,\und x)$. Finally, using again the lipschitzity of the map $s\mapsto s\varphi(s)$ and the coarea formula, we get as for  ${I_2^{ac}}$:
$$\liminf_{m\to\infty}I_2^c \geq -C\eps o_{\eps} (1).$$
Together with \eqref{auxiliary_for_ac_part}, this yields
\begin{equation}
\label{i2d}
\liminf_{k\to\infty}I_2^d \geq
-C \eps o_{\eps} (1)\,.
\end{equation}}

\medskip

{\noindent $\bullet$ {\it Step 4. Estimate of the jump part {in} $I_2$.}}

\medskip

Concerning $I_2^{j}$, we first consider its first two terms (see \eqref{def-I2}). Recalling the definition of $\z$ (for the first inequality) and \eqref{lsc},{ \eqref{recover1} and \eqref{prop-ue} (in the second inequality)}, we have
\begin{eqnarray}
\lefteqn{\int_{Q_{\rm \tau}\times \uQ_{\rm \tau}} \eta_{k,n} \d h_T^j(u, D_x u_\eps) - \int_{Q_{\rm \tau}\times \uQ_{\rm \tau}}
\eta_{k,n} T(\underline{u}) \ \nonumber\underline{\z} \cdot \d D^j_x u_\eps}
\label{hj}
\\ \dys
 &{\ge} & \int_{\uQ_{\rm \tau}} \left(\int_{Q_{\rm \tau}} \eta_{k,n} (\d h_T^j(u, D_x u_\eps) -
T(\underline{u}) {\varphi(\und u)}\ \nonumber \d | D^j_x u_\eps|) \right)\d \ux \d\underline{t}
\\ \dys
& = &  \int_{\uQ_{\rm \tau}} \left(\int_{Q_{\rm \tau}}  \eta_{k,n} \left(\int_{R_{\eps,\uu} (u)^-}^{R_{\eps,\uu} (u)^+} \nonumber (T(s)\varphi (s) -T(\uu) \varphi(\uu))\d s \right)\d \mathcal{H}^{N-1}(x){\llcorner}_{J_{R_{\eps,\uu} (u)}}  \right) 
\\ \dys  & \geq & -C\eps^2\,,
\end{eqnarray}
where in the last step we used  the mean value property as in \cite[Pag. 1388]{ACMM-MA-E}. The sum of the third and the fourth terms in $I_2$ can  be {easily seen to be nonnegative reasoning as in the previous estimate}, yielding
\begin{equation}
\label{i2j}
\liminf_{k\to\infty}I_2^j \geq -C {\eps^2 \,.} 
\end{equation}

{\noindent $\bullet$ {\it Step 5. Passing to the limit as $k\to +\infty$}}

\medskip

Combining \eqref{i2d} and \eqref{i2j} we obtain
\begin{equation}
\label{i2g}
\liminf_{k\to\infty}I_2 \geq -C \eps o_{\eps} (1)\,.
\end{equation}
{We define $\kappa_n = \tilde\rho_n \phi$ and we pass to the limit as $k\to +\infty$ in \eqref{2UE8}: in view of \eqref{i2g}, we obtain}
\begin{eqnarray}\label{remaining}
\nonumber \lefteqn{ \dys -\int_{(0,{\rm \tau})^2\times\Omega} \left({J}_{TR_{\eps,\underline{u}}}(u) (\kappa_n)_t +\nonumber
{J}_{TS_{\eps,u}}( \underline{u}) (\kappa_n)_\ut \right) \sigma}
\\ && \quad + \int_{(0,{\rm \tau})^2\times\Omega}\kappa_n S_{\eps,u}(\uu)(T(\uu)\underline\z-T(u)\z)\cdot\nabla\sigma \\ &&\nonumber \quad  + \eps  \int_{(0,{\rm \tau})^2\times\Omega} \kappa_n T(u)\z\cdot\nabla\sigma \leq C \eps o_{\eps} (1).
\end{eqnarray}

{\noindent $\bullet$ {\it Step 6. Invading $\Omega$.}}

\medskip

{We} choose a sequence $\sigma=\sigma_k \nearrow \chi_\Omega$ in \eqref{remaining}. Arguing as in the proof of Claims (10) and (11) of \cite{ACMM-MA-E} we get
\begin{eqnarray*}
\lefteqn{
\lim_{k\to\infty} \left(\int_{(0,{\rm \tau})^2\times\Omega}\kappa_n S_{\eps,u}(u)(T(\uu)\underline\z-T(u)\z)\cdot\nabla{\sigma_k}  \dys  + \eps  \int_{(0,{\rm \tau})^2\times\Omega} \kappa_n T(u)\z\cdot\nabla{\sigma_k} \right)
}
\\ \dys &=&  - \int_{{(0,{\rm \tau})^2}\times\partial\Omega} \kappa_n S_{\eps,u}(\uu)(T(\uu)[\underline\z,\nu^\Omega]-T(u)[\z,\nu^\Omega])  - \eps  \int_{{(0,{\rm \tau})^2}\times\partial\Omega} \kappa_n T(u)[\z,\nu^\Omega]\,.
\end{eqnarray*}
The passage to the limit as $k\nearrow \infty$ in the remaining terms of \eqref{remaining} is straightforward: therefore
\begin{eqnarray}
\nonumber \dys -\int_{(0,{\rm \tau})^2\times\Omega}
\left({J}_{TR_{\eps,\underline{u}}}(u) (\kappa_n)_t +\nonumber
{J}_{TS_{\eps,u}}( \underline{u}) (\kappa_n)_\ut \right)  &-& \int_{{(0,{\rm \tau})^2}\times\partial\Omega} \kappa_n S_{\eps,u}(\uu)(T(\uu)[\underline\z,\nu^\Omega]-T(u)[\z,\nu^\Omega]) \\ \dys  - \eps  \int_{{(0,{\rm \tau})^2}\times\partial\Omega} \kappa_n T(u)[\z,\nu^\Omega]&\leq& C \eps o_{\eps} (1)\,.
\label{right}
\end{eqnarray}

{\noindent $\bullet$ {\it Step 7. Conclusion.}}

\medskip

We divide {\eqref{right}} by $\eps$ and pass to the limit as $\eps\to 0$:
\begin{eqnarray*}
 &&  -\int_{{(0,{\rm \tau})^2}\times\Omega}
\left({J}_{T,\sign(\cdot-\underline{u})_+}(u) (\kappa_{n})_t +
{J}_{T,\sign(\cdot-u)_+}( \underline{u}) (\kappa_{n})_\ut \right)\\ &&
  - \int_{{(0,{\rm \tau})^2}\times\partial\Omega} \kappa_n \sign(\uu-u)_+(T(\uu)[\underline\z,\nu^\Omega]-T(u)[\z,\nu^\Omega]) -  \int_{{(0,{\rm \tau})^2}\times\partial\Omega} \kappa_n T(u)[\z,\nu^\Omega]\,\leq 0
\end{eqnarray*}
It is easy to see from \eqref{boundconddu>g}, \eqref{boundcondd=} and \eqref{boundcondd<=} that $$\sign(\uu-u)_+(T(\uu)[\underline\z,\nu^\Omega]-T(u)[\z,\nu^\Omega])\leq 0 \,,\quad \mathcal H^{N-1}-{\rm a.e. \ on \ } \partial\Omega.$$
Therefore
\begin{equation}
-\int_{{(0,{\rm \tau})^2}\times\Omega}
\left({J}_{T,\sign(\cdot-\underline{u})_+}(u) (\kappa_{n})_t +
{J}_{T,\sign(\cdot-u)_+}( \underline{u}) (\kappa_{n})_\ut \right)\\ \leq
    \int_{{(0,{\rm \tau})^2}\times\partial\Omega} \kappa_n T(u)[\z,\nu^\Omega].
\end{equation}
We divide the last equation by $b-a$ and  pass to the limit as $a\to 0$ and $b\to 0$, in this order. We obtain
\begin{equation}\label{UE8conclusion6}
-\int_{{(0,{\rm \tau})^2}\times\Omega}(
(u-\underline{u})_+(\kappa_n)_t + (\underline u-u)_+(\kappa_n)_\ut)\leq \int_{{(0,{\rm \tau})^2}\times\partial\Omega} \kappa_n [\z,\nu^\Omega]
\end{equation}
{(we used that $\z=0$ if $u=0$).}
We write
\begin{eqnarray*}
-\int_{{{(0,{\rm \tau})^2}\times\Omega}} (\underline u-u)_+\tilde\rho_n \phi'
&= & -\int_{{(0,{\rm \tau})^2}\times\Omega}
(\underline u-u)_+((\kappa_n)_t+(\kappa_n)_\ut)
\\ & \stackrel{\eqref{UE8conclusion6}}\leq & {-\int_{{(0,{\rm \tau})^2}\times\Omega} \left((\underline u-u)_+ - (u-\und u)_+ \right)(\kappa_{n})_t}+\int_{{(0,{\rm \tau})^2}\times\partial\Omega} \kappa_n [\z,\nu^\Omega]
\\
& {\stackrel{\eqref{Green}}=} & \int_{{(0,{\rm \tau})^2}\times\Omega}
(u-\underline u)(\kappa_n)_t-{\kappa_n\dive \z}
 = {\tau}\int_{{Q_{\rm \tau}}}
u(\kappa_n)_t {-\kappa_n\dive \z}
\stackrel{\eqref{dist}}= 0,
\end{eqnarray*}where we used that $\dive \z(t)\in \mathcal M(\Omega)$ for a.e. $t$.
Letting $n\to \infty$, we obtain
$$
\begin{array}{l}
- \displaystyle \int_{Q_{\rm \tau}} (
\underline u(t,x)-{u}(t, x))_+ \,  \phi'(t) \d t \d x \leq 0.
\end{array}
$$
Since this is true for all $0 \leq \phi \in {\mathcal D}((0, {\rm \tau}))$,
it implies
$$
\int_{\Omega}   (\underline u(t,x) - {u}(t, x))_+  \d x
\leq \int_{\Omega}(\underline u(0) - {u}_0)_+ \d x \ {=0} 
\ \ \ \ {\rm for \ all} \ \ t \in (0,{\rm \tau}).
$$
\end{proof}

\begin{remark}\label{remark3.14}
Let us remark the following:  as we have {already said},  our attention is focused on  the case of a  mobility given by the nonlinear term $u^{m}$. However, one might consider the case of a more general nonlinearity:
$$
\left\{
\begin{array}{ll}
u_t=\dive\left(\varphi(u)\frac{\nabla u}{|\nabla u|}\right) & \mbox{in }\ \Omega\\
u=g & \mbox{on}\  \partial\Omega\, \\ u(0)=u_0 & \mbox{in } \Omega
\end{array}
\right.
$$
where
 $\varphi(s)\in {C}([0,+\infty))$ is a strictly
{increasing
}
function. In fact, one can construct a theory and obtain existence and uniqueness of solutions. However, due to the loss of homogeneity, one cannot use Benilan-Crandall's theorem to obtain enough regularity of $u_t$ 
as the one stated in $(ii)$ of Definition \ref{def-sol}. Instead, one has to work in the dual spaces $(L^1((0,\tau);BV(\Omega)\cap L^2(\Omega) )^*$ as in \cite{ACM}, \cite{ACM_jde08} or \cite{ACMM-MA}, among others. Once one has defined the proper notion of solution, the proof of uniqueness follows exactly as in Theorem \ref{thm-uni}. However, for the existence of solutions, one has to work much harder. Moreover, without the regularity of the time derivative stated above, we cannot build a good theory on qualitative properties of the solutions.

Therefore, since our main interest in this work is to investigate the qualitative properties of the solutions to Equations \eqref{m}, and for the sake of simplicity and clarity of the presentation, we decided to present only the case of the mobility $u^m$, at the price of loosing generality.
\end{remark}

\section{Homogeneous Neumann boundary conditions}
\label{neumann}

{

The homogeneous Neumann problem,
\begin{equation}
  \label{pbparab-n0}
  \left\{\begin{array}
    {ll} u_t=\dys \dive\left(u^m\frac{\nabla u}{|\nabla u|}\right) & {\rm in \ } Q_\tau
        \\[1ex]
        u(0,x)=u_0 & {\rm in \ } \Omega
    \\[1ex]
     u^m\frac{Du}{|Du|}\cdot\nu^\Omega=0  & {\rm on \ } S_\tau,
  \end{array}\right.
\end{equation}
can be analyzed with analogous, though simpler, arguments. The notions of solution and sub-solution to problem \eqref{pbparab} are modified as follows.

\begin{definition}\label{def-sol-n0}
Let $u_0\in L^\infty_+(\Omega)$ and ${\rm \tau}< +\infty$.  A nonnegative function  $u\in C([0,{\rm \tau});$$L^1(\Omega))\cap L^\infty((0,{\rm \tau})\times\Omega)$ is:

\begin{itemize}

\item an entropy solution to \eqref{pbparab-n0} in $Q_\tau$ if $(i)$, $(ii)$, $(iii)$, and $(vi)$ in Definition \ref{def-sol} hold, the entropy inequality \eqref{main-ineq} is satisfied for any for any $\ell\in\mathcal L$ and any nonnegative $\psi \in C^\infty_c((0,{\rm \tau})\times\overline \Omega)$, and $(v)$ is replaced by
\smallskip
\begin{itemize}
\item[ $\ \qquad( v)_{N}$] for {a.e.} $t\in (0,{\rm \tau})$,
\begin{equation}\label{boundcondd-N}  [\z (t),\nu^\Omega] =0
 \quad\mbox{$\mathcal H^{N-1}$-a.e. on $\partial\Omega$;}
 \end{equation}
\end{itemize}
\smallskip

\item an entropy solution to \eqref{pbparab-n0} in $Q$ if it is an entropy solution to \eqref{pbparab-n0} in $Q_{\rm \tau}$ for all ${\rm \tau}$;

\item an entropy sub-solution to \eqref{pbparab-n0} in $Q_{\rm \tau}$ if: $(i)$ and $(ii)$ in Def. \ref{def-sol} hold; $(iii)_{sub}$ and $(vi)_{sub}$ in Def. \ref{def-sub} hold; the entropy inequality \eqref{main-ineq} is satisfied for any for any $\ell\in\mathcal L$ and any nonnegative $\psi \in C^\infty_c((0,{\rm \tau})\times\overline \Omega)$; $(v)_N$ holds;

\item an entropy subsolution to \eqref{pbparab-n0} in $Q$ if it is an entropy subsolution to \eqref{pbparab-n0} in $Q_{\rm \tau}$ for all ${\rm \tau}$.
\end{itemize}

\end{definition}

Using the analysis of the resolvent equation for \eqref{pbparab-n0} contained in \cite[Section 7]{GMP}, the following existence, uniqueness, and comparison results can be proved:
%
%

\begin{theorem}\label{thm-exi-uni-n0}
\begin{itemize}

Let $u_0\in L^\infty_+(\Omega)$ and $\tau\in (0,+\infty]$.

\item  There exists an entropy solution of \eqref{pbparab-n0} in $Q_\tau$ in the sense of Definition \ref{def-sol-n0}.

\item if $u$, resp. $\underline u$, are an entropy solution, resp. sub-solution, to \eqref{pbparab-n0} in $Q_{\tau}$, then $\underline u(t)\le u(t)$ for all $t\in (0,{\rm \tau})$. In particular, the entropy solution $u$ is unique.

\end{itemize}
\end{theorem}

The proof of Theorem \ref{thm-exi-uni-n0} closely follows the lines of that of Theorems \ref{thm-exi} and \ref{T-sub}, with many simplifications due to the homogeneous Neumann boundary conditions. We only mention that one has to use the existence and uniqueness result in \cite[Theorem 7.2]{GMP} for the corresponding resolvent equation. The estimates and the passage to the limit are completely analogous, in fact simpler, due to the absence of boundary terms: for instance, the boundary condition \eqref{boundcondd-N} follows directly from \eqref{trace-conv}, and in the proof of Lemmas \ref{lem:ubv} and \ref{2.9} one has to use lower semi-continuity of the functional
$$
u\in L^1(\Omega)\mapsto \left\{\begin{array}{cc}\displaystyle \int_\Omega \psi d |D\phi_\ell(u)| & {\rm if \ } u\in TBV(\Omega) \\ +\infty & {\rm otherwise} \end{array}\right., {\rm \ with \ }0\leq \psi\in \mathcal D(\Omega)\,,
$$
(see \cite[Theorem 3.1]{AdCF_esaim07}) which does not contain any boundary contribution.

}

\section{{Self-similar solutions and the finite speed of propagation property}
}

\subsection{Self-similar source-type solutions}

{Due to its homogeneity, \eqref{m} possesses a two-parameter family (besides translations in time and space) of self-similar source type solutions: they are supported on moving balls and thereon spatially constant. }

\begin{theorem}\label{th:ss}
{Let $x_0\in \Omega$, $ \rm X>0$, $t_0>0$ and ${\rm T}>0$ be such that ${\rm X}^{-1}(\alpha^{-1}t_0{\rm T})^\alpha\subset \Omega$. Then the function} 
\begin{equation}
\label{ss}
u_{\rm  s}(t,x) = {\rm T}^{\frac{1}{m-1}-\alpha N} {\rm X}^{-\frac{1}{m-1}} (r(t))^{-N}\chi_{B_t}, \ \ B_t:=B(x_0,{\rm X}^{-1} {\rm T}^\alpha r(t)),
\end{equation}
with
\begin{equation}\label{def-rt}
 \quad r(t)=(\alpha^{-1}(t_0+t))^\alpha,\quad \alpha=\frac{1}{N(m- 1)+1},
\end{equation}
{is {an entropy} solution to both \eqref{pbparab} with $g=0$ and \eqref{pbparab-n0} in $(0,\tau)\times \Omega$, where $\tau= \sup\{t>0: \ B_t{\Subset}\Omega\}$.}
%
\end{theorem}

\begin{proof}
By translation invariance in space and time, and by the scaling invariance
\begin{equation}
\label{note-inv}
(t,x,u)\mapsto (Tt,{\rm X}x,({\rm X}/{\rm T})^{{1/(m-1)}} u),
\end{equation}
it suffices to consider the case ${\rm X}=1$, ${\rm T}=1$, $x_0=0$, $t_0={1}$: we thus look for solutions of the form
$$
u(t,x) = (r(t))^{-N}\chi_{B_t},\quad  B_t:=B(0,r(t))\,,
$$
with $r$ to be characterized below. Define
$$
\w= \left\{\begin{array}{ll} -\frac{x}{r(t)} & x\in B_t \\ -\frac{x}{|x|} & x\in \Omega\setminus B_t,\end{array}\right. \quad\mbox{hence}\quad
\z=u^m \w = {-x (r(t))^{- m N -1}\chi_{B_t}}
$$
Then
$$
\dive \z = (r(t))^{- m N} \mathcal H^{N-1}\res \partial B_t { -N (r(t))^{-m N -1}\chi_{B_t}\mathcal L^N.}
$$
On the other hand, it is easily computed
$$
u_t=  (r(t))^{-N}r'(t) \mathcal H^{N-1}\res \partial B_t- N (r(t))^{-N-1} r'(t) \chi_{B_t}\mathcal L^N.
$$
Hence \eqref{dist} holds if and only if
\begin{equation}
\label{ode}
(r(t))^{(1- m)N}=r'(t),
\end{equation}
{which implies \eqref{def-rt} with $t_0=1$.} 
%
%
In view of the form of $u$ and $\z$, the entropy condition decouples into two inequalities between measures for the Lebesgue, resp. the jump parts:
\begin{eqnarray}\label{decoupling-ac}
|\nabla \Phi_\ell(u)| &\le& -(J_\ell(u))_t + (\dive(\ell(u)\z))^{ac},
\\  \label{decoupling-j}
|D^j \Phi_\ell(u)| &\le & -D^j_t(J_\ell(u)) + (\dive(\ell(u)\z))^{j}
\end{eqnarray}
for any $\ell \in \mathcal L$, where we recall that
$$
\Phi_\ell(u)=\int_0^u \ell'(\sigma)\sigma^{m}\d \sigma, \quad  J_\ell(u)=\int_0^u \ell(\sigma)\d \sigma.
$$
{Inequality \eqref{decoupling-ac}} is satisfied as an equality in view of \eqref{dist}. {Indeed,} by integration by parts and {the} chain's rule, $$-(J_\ell(u))_t + (\dive(\ell(u)\z))^{ac}= \ell(u)u_t^{ac}+\ell(u)\dive(\z)^{ac}+\ell'(u)\z\cdot\nabla u \stackrel{\eqref{dist}}=\ell'(u)\z\cdot \nabla u,
$$
{whence \eqref{decoupling-ac} since} $\nabla u\equiv 0$.

On the other hand, arguing as in \cite{GMP1} (see the proof of Proposition 4.1, in particular (4.10)), {\eqref{decoupling-j}} reduces to
\begin{equation}\label{lk}
\int_{0}^{u^+}(\ell'(\sigma)\sigma(\sigma^{m-1} - r') \d \sigma \le  u^+(t) \ell(u^+)((u^+)^{m-1}-r'),
\end{equation}
where $u^+= (r(t))^{-N}$. In view of \eqref{ode}, $(u^+)^{m-1}=r'$: hence the right-hand side of \eqref{lk} is zero and the left-hand side is negative. {Therefore $u$ is an entropy solution to \eqref{pbparab} as long as its support is contained in $\Omega$, and \eqref{ss} follows from scaling.}
\end{proof}

\subsection{The finite speed of propagation property.}

{It follows immediately from Theorem \ref{th:ss} and comparison that solutions to \eqref{m} enjoy the finite speed of propagation property: in words, a compactly supported  initial datum induces a solution whose support remains compact for any later time, with a universal control on its width. }

\begin{theorem}\label{teo:fsp}
{Let $u$ be an entropy solution to \eqref{pbparab} with $g=0$ or to \eqref{pbparab-n0},}
such that {\rm supp}$(u_0)\subset B(x_0,R)\Subset\Omega$, and let $d=$dist$(B(x_0,R),\partial\Omega)$. Then
$$
{\rm supp} u(t,\cdot)\subset B\left(x_0,R\left(1+\alpha^{-1}t\right)^\alpha\right) \quad\mbox{as long as $\ R\left(1+\alpha^{-1}t\right)^\alpha<R+d$.}
$$
\end{theorem}

{Note that the speed of propagation is independent of any norm of $u$: it just depends on the width of the initial support. This is quite natural, in view of the scaling invariance \eqref{note-inv}.--}

\begin{proof}
By translation invariance, we may assume without loss of generality that $x_0=0$. Choose $x_0=0$ and $t_0=\alpha$, so that $r(0)=1$, in the definition \eqref{ss} of $u_{\rm s}$. We require
$u(0,x)\le u_{\rm s}(0,x)$, which is implied by
$$
\|u_0\|_\infty \chi_{B(0,R)} \le u_{\rm s}(0,x)= {\rm T}^{\frac{1}{m-1}-\alpha N} {\rm X}^{-\frac{1}{m-1}} \chi_{B(0,{\rm X}^{-1} {\rm T}^\alpha)}.
$$
Therefore, we choose ${\rm X}$ and ${\rm T}$ such that
$$
\|u_0\|_\infty = {\rm T}^{\frac{1}{m-1}-\alpha N} {\rm X}^{-\frac{1}{m-1}} \quad\mbox{and}\quad R= {\rm X}^{-1} {\rm T}^\alpha.
$$
By the comparison given in  Theorem \ref{T-sub}, $u\le u_{\rm s}$ as long as $\supp u_{\rm s}\subset\Omega$, i.e. ${\rm X}^{-1} {\rm T}^\alpha r(t)=Rr(t)<R+d$.
\end{proof}

\section{{The Cauchy problem}
}

\subsection{Existence and uniqueness of solutions.}

Let $u_0\in L^{\infty}_{loc}(\mathbb{R}^N)$  be nonnegative. We consider the Cauchy problem
\begin{equation}
  \label{pbparab-C}
  \left\{\begin{array}
    {cc} u_t=\dys \dive\left(u^m\frac{\nabla u}{|\nabla u|}\right) & {\rm in \ }{(0,{\rm \tau})\times}\mathbb{R}^N
        \\ \\ u(0,x)=u_0 & {\rm in \ } \mathbb{R}^N\,.
  \end{array}\right.
\end{equation}

\begin{definition}\label{def-sol_c}
Let $u_0\in L^\infty_{loc} (\R^N)$  be nonnegative {and ${\rm \tau}<+\infty$}. A nonnegative function $u\in C([0,{\rm \tau}); L^1_{loc}(\R^N))\cap L^\infty_{loc}([0,{\rm \tau}]\times\R^N)$ is an entropy solution to
{\eqref{pbparab-C} in $(0,{\rm \tau})\times \R^N$}
if:
\begin{itemize}
\item[$(i)$] $\ell(u)\in L^1_{}([0,{\rm \tau}); BV_{loc}(\R^N))$ for all $\ell\in\mathcal L$;
\item[$(ii)$]  $u_{t}\in L^{\infty}_{loc,  w} ((0,{\rm \tau}], \mathcal M_{loc} (\R^N))$
\item[$(iii)$] There exists $\w\in L^{\infty}((0,{\rm \tau})\times\R^N)$ such that $\|\w\|_{\infty}\leq 1$ with $\z:=\vp (u)\w$ satisfying
    \begin{equation}\label{dist-C}
    {u_t(t)=\dive \z(t) \quad\mbox{as distributions for a.e. $t\in (0,{\rm \tau})$;}}
    \end{equation}
\item[$(iv)$]  the {e}ntropy inequality
\begin{equation}\label{main-ineq-C}
\int_{0}^{{\rm \tau}}\int_{\R^N} \psi\ \d h (u, D\ell (u)) \leq  \int_{0}^{{\rm \tau}} \int_{\R^N}  J_{\ell}(u) \psi_{t}  - \int_{0}^{{\rm \tau}}\int_{\R^N} \ell(u)\z\cdot\nabla \psi
\end{equation}
 holds for any $\ell\in\mathcal L$, and any nonnegative $\psi \in C^\infty_c((0,{\rm \tau})\times \R^N)$;
\item[$(v)$] $u(0)=u_0\ \ \mbox{in $L^{1}_{loc}(\R^N)$}\,.$
\end{itemize}
{A nonnegative function $u$ is an entropy solution to \eqref{pbparab-C} in $(0,+\infty)\times \R^N$ if it is an entropy solution to to \eqref{pbparab-C} in $(0,{\rm \tau})\times \R^N$ for all ${\rm \tau}>0$.}
\end{definition}

Definition \ref{def-sol_c} implies mass conservation if $u_0\in {L^1_+}(\R^N)$:
\begin{proposition}\label{prop-mass-cons}
{
Let ${\rm \tau}\le +\infty$. If $u_0\in L^1_+(\R^N)$, the entropy solution to \eqref{pbparab-C} in $(0,{\rm \tau})\times \R^N$ is such that
}
\begin{equation}\label{eq-mass-cons}
\int_{\R^N} u(t,x)\d x =\int_{\R^N} u_0(x) \d x \quad {\rm for \ all \ }t\in {(0,{\rm \tau})}.
\end{equation}
\end{proposition}
The proof is exactly the same as the proof of  \cite[Proposition 2.3]{GMP1}, hence we omit it. It is also easy to check that:
\begin{proposition}
The self-similar source-type solutions in Theorem \ref{th:ss} solve {\eqref{pbparab-C} in $(0,+\infty)\times \R^N$}.
\end{proposition}

{In view of the uniform bound on the support given by Theorem \ref{teo:fsp}, entropy solutions to \eqref{pbparab-C} for a generic, bounded initial datum with compact support can be obtained in a standard way, gluing together those of the homogeneous Dirichlet or Neumann problem:}

\begin{theorem}
Let ${0\leq }u_0\in L^\infty_{loc} (\R^N)$ with compact support. Then there exists {an entropy solution to \eqref{pbparab-C} in $(0,+\infty)\times \R^N$.}
\end{theorem}

{
\begin{definition}
The definition of subsolution is the same as Definition \ref{def-sol_c}, except that the {equalities} in \eqref{dist-C} and in item $(v)$ have to be replaced by a less than or equal sign.
\end{definition}
}

{With this notion at hand, we can formulate } 
the following comparison principle, leading to uniqueness of solutions:

\begin{theorem}\label{T-sub-C}
{Let ${\rm \tau}>0$ and $u_0\in L_{loc}^\infty(\R^N)$ be nonnegative.} {Let $u$ and $\und u$ be an entropy solution, respectively subsolution, to \eqref{pbparab-C} in $(0,{\rm \tau})\times \R^N$}
such that supp $\underline u\cap ((0,{\rm \tau})\times\R^N)$ is compact. Then $\underline u(t)\leq u(t)$ for all $t\in (0,{\rm \tau})$.
\end{theorem}

\begin{proof}
  The proof closely follows that of Theorem \ref{T-sub}, and is in fact simpler. We assume all the notation therein. {After repeating line by line the arguments up to Step 5, we arrive at a formula identical to \eqref{remaining}:
\begin{eqnarray}
\nonumber \lefteqn{ \dys -\int_{(0,{\rm \tau})^2\times\Omega} \left({J}_{TR_{\eps,\underline{u}}}(u) (\kappa_n)_t +\nonumber
{J}_{TS_{\eps,u}}( \underline{u}) (\kappa_n)_\ut \right) \sigma}
\\ \nonumber && \quad + \int_{(0,{\rm \tau})^2\times\Omega}\kappa_n S_{\eps,u}(\uu)(T(\uu)\underline\z-T(u)\z)\cdot\nabla\sigma \\ &&+ \eps  \int_{(0,{\rm \tau})^2\times\Omega} \kappa_n T(u)\z\cdot\nabla\sigma  \label{remaining-C}\nonumber \leq C \eps o_{\eps} (1).
\end{eqnarray}
}
Dividing \eqref{remaining-C} by $\eps$ and passing to the limit as $\eps\to 0^+$ we get
\begin{eqnarray}\label{remaining-C1}
\nonumber  \lefteqn{ -\int_{(0,{\rm \tau})^2\times\R^N}
\left({J}_{T{\rm sign }{(\cdot-\underline{u})_+}}(u) (\kappa_n)_t +
{J}_{T{\rm sign }(\cdot -u)_+}( \underline{u}) (\kappa_n)_\ut \right) \sigma}
\\  &+& \int_{(0,{\rm \tau})^2\times\R^N}\kappa_n \chi_{\{\underline u> u\}}(T(\uu)\underline\z-T(u)\z)\cdot\nabla\sigma \\ && \nonumber+ \int_{(0,{\rm \tau})^2\times\R^N} \kappa_n T(u)\z\cdot\nabla\sigma\leq 0.
\end{eqnarray}
Since the support of $\uu$ is compact, we may choose $\sigma$ as a cut-off function such that $\sigma\equiv 1 $ on the support of $\underline u$. Observe that ${\{\underline u>u\}}\subset {\rm supp} (\underline u)$, then \eqref{remaining-C1} turns into
\begin{equation*}\label{remaining-C2}
-\int_{(0,{\rm \tau})^2\times\R^N}
\left({J}_{T{\rm sign }{(\cdot-\underline{u})_+}}(u) (\kappa_n)_t + {J}_{T{\rm sign }(\cdot -u)_+}( \underline{u}) (\kappa_n)_\ut \right) \sigma +  \int_{(0,{\rm \tau})^2\times\R^N} \kappa_n T(u)\z\cdot\nabla\sigma\leq 0.
\end{equation*}
From here, the proof continues as that of Theorem \ref{T-sub}.
\end{proof}

\subsection{Characterization of solutions: the Rankine-Hugoniot condition}

Assume that $u\in BV_{loc}((0, \tau )\times \R^N)$. Let us denote by $ {J}_u$ the jump set of $u$ as a function of $(t, x)$.
Let $\nu := \nu_u = (\nu_t , \nu_x)$ be the unit
normal to the jump set of $u$ so that $D^j_{t,x}u=[u]\nu\mathcal H^{N}\res_{J_u}$.

\begin{lemma}\label{lemma_caselles}
  \cite[Lemma 6.6, Proposition 6.8]{Caselles_jde11} Let $u\in BV_{loc}((0,\tau)\times\R^N)$, let $\z\in L^\infty([0,\tau]\times\R^N;\R^N)$ be such that $u_t=\dive\z$ 
  ,
  and let the speed of the discontinuity set be defined by
  $$
  v(t,x):=\frac{\nu_t(t,x)}{|\nu_x(t,x)|} \quad \mbox{$\mathcal H^{N}$-a.e. on $J_u$.}
  $$
  Then
  \begin{equation}\label{c11}
  \nu_t \mathcal H^{N}\res_{J_u}={v}\mathcal H^{N-1}\res_{J_{u(t)}} dt
  \end{equation}
   and
  $$
  [u(t)]v(t)=[\z,\nu^{J_{u(t)}}]^+-[\z,\nu^{J_{u(t)}}]^-\quad \mbox{$\mathcal H^{N-1}$-a.e. on $J_{u(t)}$}.
  $$
\end{lemma}
We have the following characterization of entropy solutions to Problem \ref{pbparab-C}:
\begin{theorem}\label{thm:charac}
  Let $u\in C([0,{\rm \tau}); L^1_{loc}(\R^N))\cap L^\infty_{loc}([0,{\rm \tau}]\times\R^N)$ satisfy $(i)$-$(iii)$ in Definition \ref{def-sol_c}. Then, the entropy condition \eqref{main-ineq-C} is satisfied iff
  \begin{equation}\label{conddiff}
  \z\cdot\nabla u=|\nabla \Phi (u)|\quad\mbox{and}\quad |D^c\Phi_\ell(u)|\leq -(J_\ell(u))_t^c+(\dive \ell(u) \z)^c {\quad\mbox{$\forall\ell\in\mathcal L$}}
  \end{equation}
   and
   \begin{equation}
    \label{conditionatjump} [\z,\nu^{J_{u(t)}}]^\pm=(u^m)^\pm {\rm \ sign \ }(u^+-u^-)\quad \mbox{$\mathcal H^{N-1}$-a.e. on $J_{u(t)}$ {for a.e. $t\in (0,\tau)$}}.
  \end{equation}
Moreover, a.e. $t\in [0,\tau]$, it holds
$$
v(t,x)=\frac{(u^+)^m-(u^-)^m}{u^+-u^-}\quad \mbox{$\mathcal H^{N-1}$-a.e. on $J_{u(t)}$}.
$$
\end{theorem}
\begin{proof}
  Observe first that the entropy condition decouples into three inequalities between measures for the Lebesgue, Cantor, and jump parts{, respectively:}
\begin{eqnarray}
\label{rh-ac}
|\nabla \Phi_\ell(u)| & \le & -(J_\ell(u))_t + (\dive(\ell(u)\z))^{ac},
\\ \label{rh-c}
|D^c \Phi_\ell(u)| & \le & -(J_\ell(u))_t^c + (\dive(\ell(u)\z))^{c},
\\ \label{rh-j}
|D^j \Phi_\ell(u)| & \le & -D^j_t(J_\ell(u)) + (\dive(\ell(u)\z))^{j}
\end{eqnarray}
for all $\ell \in \mathcal L$ {and a.e. $t\in(0,\tau)$}. {Arguing} as in the proof of Theorem \ref{th:ss}, {\eqref{rh-ac} is easily seen to be} equivalent to the following {in}equality for any $\ell\in\mathcal L$:
$$
\ell'(u)\varphi(u)|\nabla u|=|\nabla \Phi_\ell(u)|\leq \ell'(u)\z\cdot\nabla u{.} 
$$
{Since by $(iii)$ $\z\cdot\nabla u\le \varphi(u)|\nabla u|$, \eqref{rh-ac} holds if and only if} $\z\cdot\nabla u=|D\Phi(u)|${, i.e. \eqref{conddiff}$_1$,} holds.
{Since \eqref{conddiff}$_2$ coincides with \eqref{rh-c}, it remains to prove that \eqref{rh-j} is equivalent to \eqref{conditionatjump}. In view of \eqref{c11}, \eqref{rh-j}}
%
is equivalent to
\begin{equation}\label{conditionjumpentropy}
[\Phi_\ell(u)]+[J_\ell(u)]{v}\leq [\ell(u)\z,\nu^{J_{u(t)}}]^+-[\ell(u)\z,\nu^{J_{u(t)}}]^{-} \quad \mbox{$\mathcal H^{N-1}$-a.e. on  $J_{u(t)}$.}
\end{equation}

Assume that \eqref{conditionjumpentropy} holds for any $\ell\in\mathcal L$ {and a.e. $t\in(0,\tau)$.} If $u^+>u^-$ (the other case is analogous), we let
$$
\ell_\eps(s):=\frac{1}{\eps}(s-u^{-})\chi_{[u^-,u^-+\eps]}+\chi_{]u^-+\eps, u^+-\eps]}+\left(2+\frac{1}{\eps}(s-u^+)\right)\chi_{[u^+-\eps,u^+]}+2\chi_{[u^+,\infty[}.
$$
Then, taking $\eps\to 0^+$ in \eqref{conditionjumpentropy}, we obtain that
{
$$
[\Phi_{\ell_\eps}(u)]= \int_{u^-}^{u^+} \ell_\eps'(\sigma)\sigma^m \d \sigma \stackrel{\eps\to 0}\to  2(u^+)^m-m\int_{u^-}^{u^+}\sigma^{m-1}\d \sigma = (u^+)^m + (u^-)^m,
$$
whence
}
$$
(u^+)^m + (u^-)^m+v[u]\leq 2[\z,\nu^{J_{u(t)}}]^+
$$
which, in view of Lemma \ref{lemma_caselles}, yields
$$
(u^+)^m+(u^-)^m\leq [\z,\nu^{J_{u(t)}}]^++[\z,\nu^{J_{u(t)}}]^-\leq (u^+)^m+(u^-)^m.
$$ Therefore \eqref{conditionatjump} holds.

Suppose now that \eqref{conditionatjump} holds, and suppose that we are again in a jump point where $u^-(t,x)<u^+(t,x)$. Then, \eqref{conditionjumpentropy} reads as:
$$\int_{u^-}^{u^+}\ell'(\sigma)\sigma^m \,d\sigma+ v\int_{u^-}^{u^+} \ell(\sigma)\, d\sigma\leq \ell(u^+)(u^+)^m-\ell(u^-)(u^-)^m.$$Integrating by parts the first term and using Lemma \ref{lemma_caselles}, then we will have to show that
$$ \int_{u^-}^{u^+}\ell(\sigma)\left(\frac{(u^+)^m-(u^-)^m}{u^+-u^-}-m\sigma^{m-1}\right)\,d\sigma\leq 0,$$but this inequality is {trivially satisfied} by the convexity of $\sigma\mapsto \sigma^m$.
\end{proof}

\section{Waiting-time solutions}

\subsection{Explicit solutions.}
We construct a family of solutions which exhibit a waiting-time phenomenon. {As a byproduct we infer that the operator $\mathcal A_g$ is not completely accretive; this  in contrast to the case $m = 1$, see  \cite{ACMM-JEE}}.

\begin{proposition}\label{prop:wts}
Let $x_0\in \R^N$, ${D}_0>0$, $C_0>0$, $R>0$, and $\rho_0\in (0,R)$. Consider $B_\rho:=B(x_0,\rho)$ and let
$$
p:= \frac{N(m-1)+1}{m}> 1,\quad { \rm \tau^\ast}:=\frac{\rho_0 D_0^{1-m}} {{mp}}\left(\left(\frac{R}{\rho_0}\right)^p-1\right).
$$
There exist:
\begin{itemize}
\item[-] an increasing function $\rho\in C([0,{\rm \tau^*}])$ such that $\rho(0)=\rho_0$ and $\rho({\rm \tau^*})=R$,
\item[-] a decreasing function $D\in C([0,{\rm \tau^*}))$ such that $D(0)=D_0$, $D({\rm \tau^*})>0$,
\end{itemize}
such that the function
$$
u(t,x)= \left\{\begin{array}{ll}
D(t) \chi_{B_{\rho(t)}} + C(t)\psi(r)\chi_{B_R\setminus B_{\rho(t)}} \quad & t<{\rm \tau^*}
 \\[1ex] \dys D({\rm \tau^*}) \left(\frac{t}{{\rm \tau^*}}\right)^{-\frac{N}{mp}} \chi_{B_{R(t/{\rm \tau^*})^{1/{mp}}}} & t>{\rm \tau^*}
\end{array}\right., \quad \quad r={\|x-x_0\|}
$$
is a solution to the Cauchy problem \eqref{pbparab-C}, where
\begin{eqnarray*}
(C(t))^{1-m} = C_0^{1-m}\left(1-\frac{t}{{\rm \tau^*}}\right),
\quad \psi(r)=  \frac{D_0}{C_0} \left(\frac{r}{\rho_0} \frac{\left(\left(\frac{R}{r}\right)^p-1\right)}{ \left(\left(\frac{R}{\rho_0}\right)^p-1\right)}\right)^{1/(m-1)}.
\end{eqnarray*}
\end{proposition}

\begin{proof}
Without loss of generality, we set $x_0=0$. Let us first consider $t<{\rm \tau^*}$. We require initial conditions,
$$
D(0)=D_0, \ C(0)=C_0, \ \rho(0)=\rho_0\in (0,R),
$$
and $u$ to be continuous in $\R^N$,
\begin{equation}\label{cont}
D(t)=C(t)\psi(\rho(t)), \quad \psi(R)=0.
\end{equation}
Since $u$ is supported in $B_R$, it suffices to perform the analysis there. Define
$$
\w= \left\{\begin{array}{ll} -\dfrac{x}{\rho(t)} & x\in B_{\rho(t)} \\[2ex] -\dfrac{x}{r} & x\in B_R\setminus B_{\rho(t)},\end{array}\right. \ \mbox{hence}\quad
\z=u^m \w = \left\{\begin{array}{ll} -\dfrac{(D(t))^{m}}{\rho(t)}x & x\in B_{\rho(t)} \\[2ex] -\dfrac{(C(t)\psi(r))^m}{r} x & x\in B_R\setminus B_{\rho(t)}.
\end{array}\right.
$$
Then
$$
\dive \z = \left\{\begin{array}{ll} -\dfrac{N(D(t))^{m}}{\rho(t)} & x\in B_{\rho(t)} \\[2ex] -m(\psi(r))^{m-1}(C(t))^m\psi'(r) -(N-1)\dfrac{(C(t)\psi(r))^m}{r} & x\in B_R\setminus B_{\rho(t)}
\end{array}\right.
$$
On the other hand, in view of \eqref{cont},
$$
u_t= \left\{\begin{array}{ll} D'(t) & x\in B_{\rho(t)} \\[2ex] C'(t)\psi(r) & x\in B_R\setminus B_{\rho(t)},\end{array}\right.
$$
Therefore we obtain the conditions
$$
D'(t)= -\dfrac{N(D(t))^{m}}{\rho(t)}
$$
{and, by separation of variables,}
\begin{equation}\label{AP}
(C(t))^{-m}C'(t) = - m(\psi(r))^{m-2}\psi'(r) -(N-1)\dfrac{(\psi(r))^{m-1}}{r}=K\,,
\end{equation}
where $K$ is a constant to be determined later.
An integration using $\psi(R)=0$ and initial conditions yields
\begin{eqnarray}\label{AA}
(D(t))^{1-m} &=& D_0^{1-m} +N(m-1)\int_0^t\frac{1}{\rho(t')}\d t',
\\ \label{AC}
(C(t))^{1-m} &=& C_0^{1-m} -(m-1)Kt, \quad t<\frac{C_0^{1-m}}{K(m-1)},
\\ \label{AP2}
(\psi(r))^{m-1}&=&\frac{K(m-1)}{mp} r \left(\left(\frac{R}{r}\right)^p-1\right),  \quad K\in \R.
\end{eqnarray}
Condition \eqref{cont} at $t=0$ determines
$$
K:=\left(\frac{D_0}{C_0}\right)^{m-1}\frac{mp}{(m-1)\rho_0}\left(\left(\frac{R}{\rho_0}\right)^p-1\right)^{-1}.
$$
In order to determine $\rho$, we rewrite \eqref{cont} as
$
D^{1-m}=C^{1-m}(\psi(\rho))^{1-m},
$
differentiate it in time,
$$
(D^{1-m})'=(C^{1-m})'(\psi(\rho))^{1-m}+C^{1-m}(\psi^{1-m})'\rho',
$$
note that
$$
\label{APN}
(\psi^{1-m})' = \frac{m-1}{m} \psi^{1-m} \frac{-m\psi'}{\psi} \stackrel{\eqref{AP}}=  \frac{m-1}{m} \psi^{1-m} \left(\frac{N-1}{r} + K \psi^{1-m}\right),
$$
and substitute using \eqref{AA}, \eqref{AC}, and \eqref{AP}:
$$
\frac{N}{\rho} =- K (\psi(\rho))^{1-m} +C^{1-m} \frac{1}{m} (\psi(\rho))^{1-m} \left(\frac{N-1}{\rho} + K (\psi(\rho))^{1-m}\right)\rho',
$$
i.e.
$$
\frac{m}{C_0^{1-m} -(m-1)Kt} = (\psi(\rho))^{1-m} \frac{\frac{N-1}{\rho} + K (\psi(\rho))^{1-m}}{\frac{N}{\rho} + K (\psi(\rho))^{1-m}} \rho',
$$
a separable ODE which has a unique, strictly increasing solution starting from $\rho(0)=\rho_0$, defined for $t<{\rm \tau^*}$ and such that $\rho(t)\to R$ as $t\to {\rm \tau^*}$. As $t\to {\rm \tau^*}$, we have
$$
u({\rm \tau^*},x)\to D({\rm \tau^*}) \chi_{B_{R}}.
$$
Finally, we note that, by Proposition \ref{prop-mass-cons}, $D(\tau^*)>0$.

For $t\ge \rm \tau^*$, we observe that $u$ coincides with one of the self-similar solutions $u_{\rm s}$ for suitable values of the scaling parameters in Theorem \ref{th:ss}. This completes the proof.
\end{proof}

\begin{figure}[H]
\begin{minipage}[c]{0.45\textwidth}
\includegraphics[keepaspectratio,height=6.5cm, width=6.5cm]{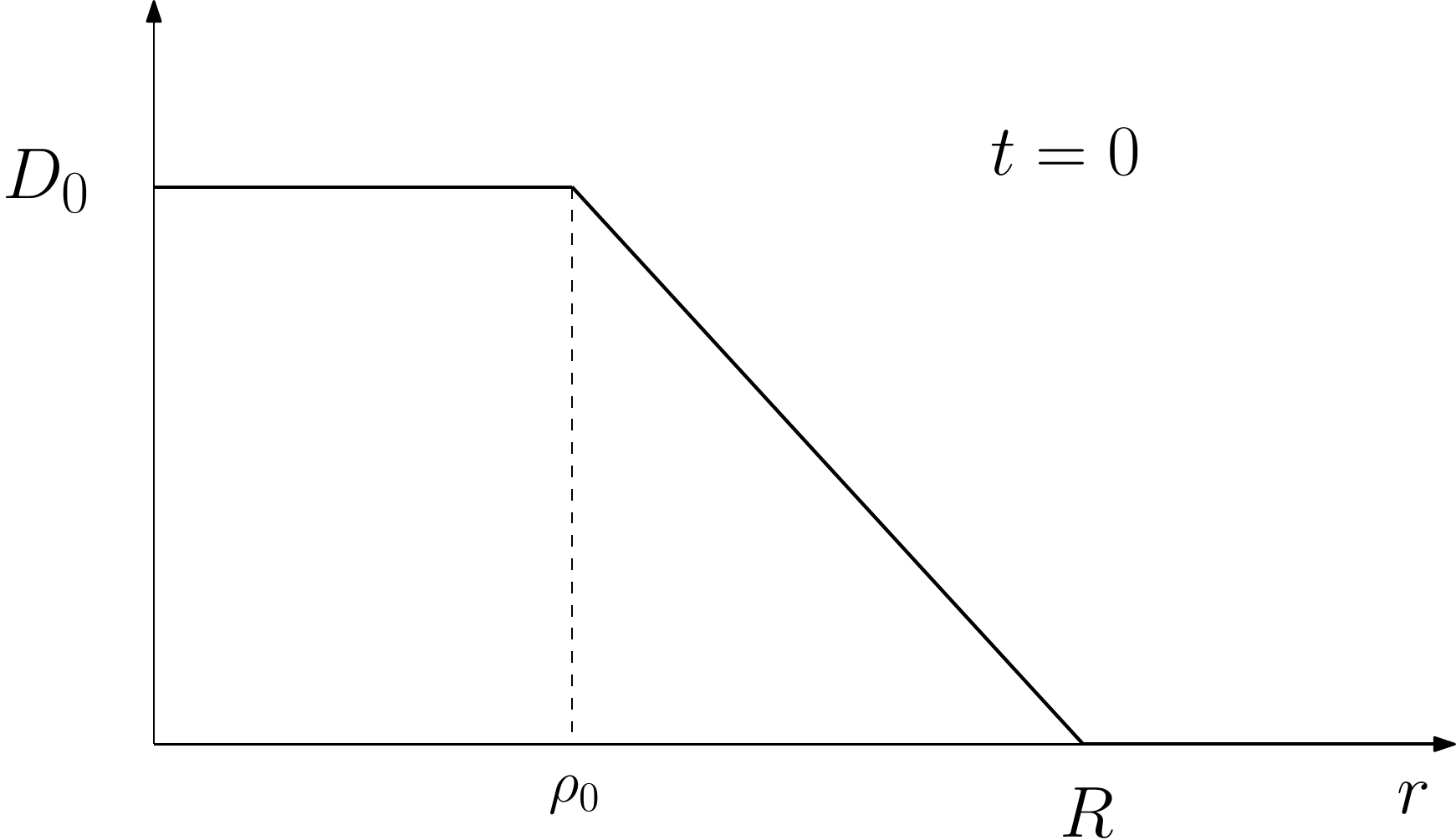}\end{minipage}
\begin{minipage}[c]{0.45\textwidth}\includegraphics[keepaspectratio,height=6.5cm, width=6.5cm]{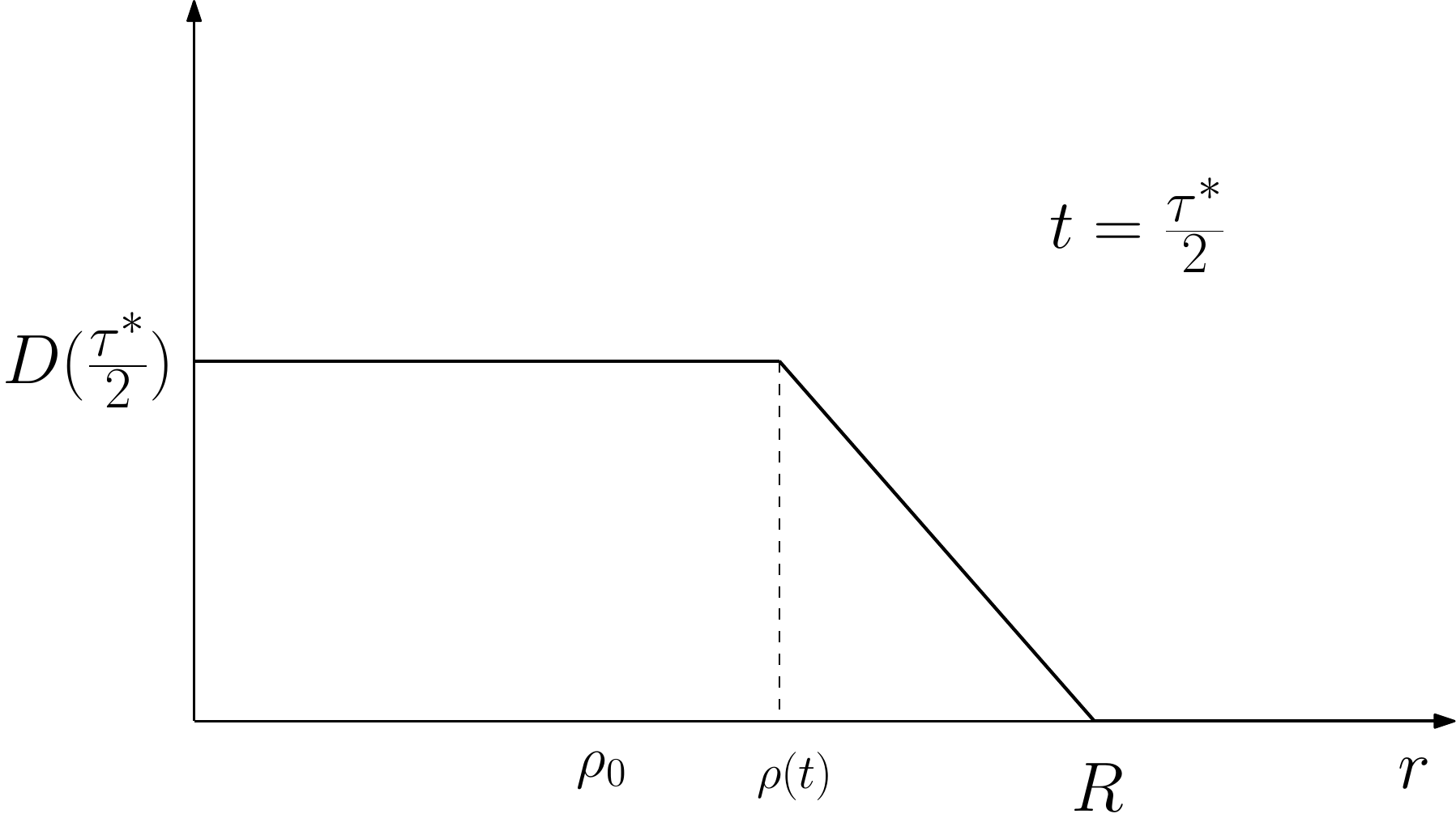} \end{minipage}\\
\begin{minipage}[c]{0.45\textwidth}\includegraphics[keepaspectratio,height=6.5cm, width=6.5cm]{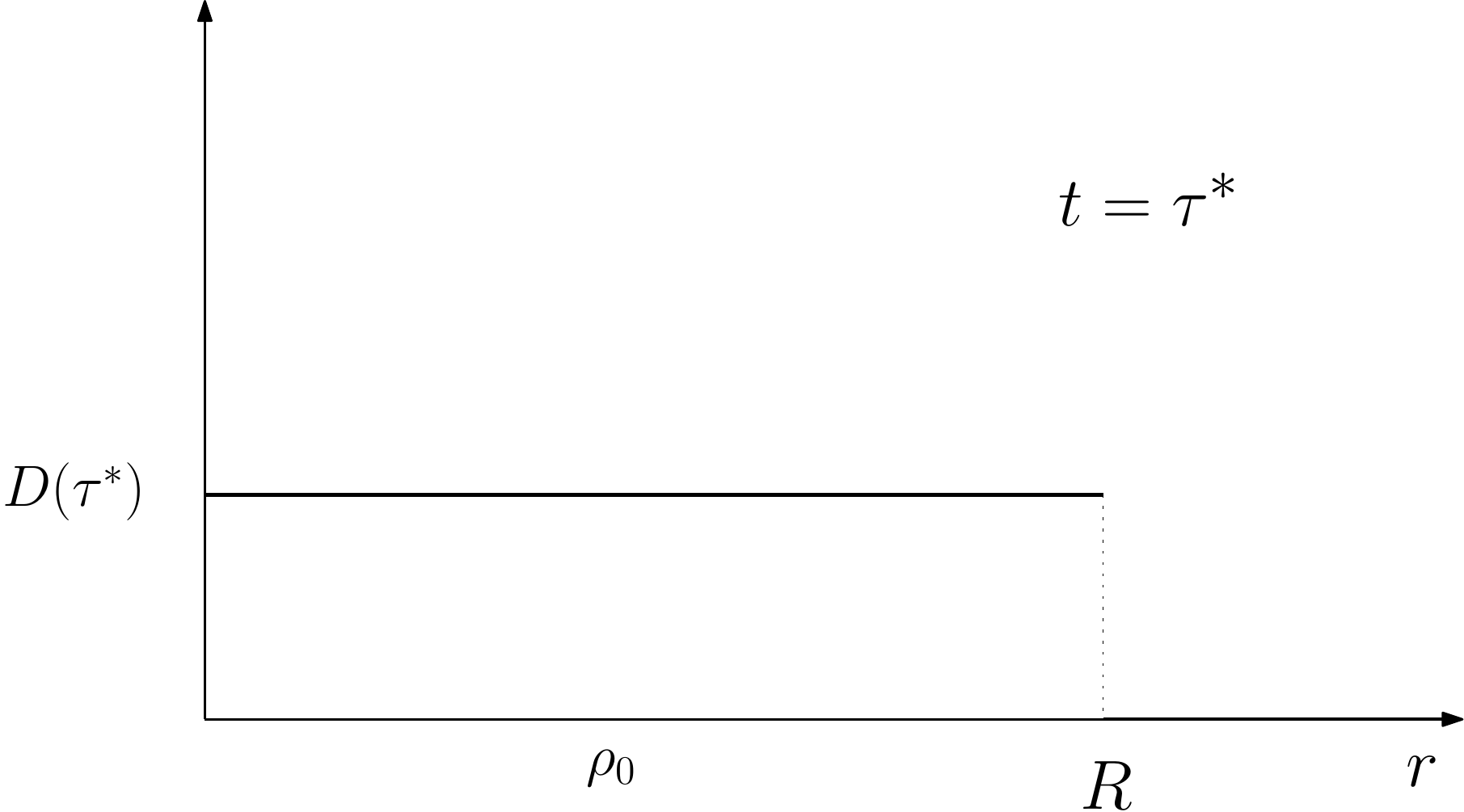} \end{minipage}
\begin{minipage}[c]{0.45\textwidth}\includegraphics[keepaspectratio,height=6.5cm, width=6.5cm]{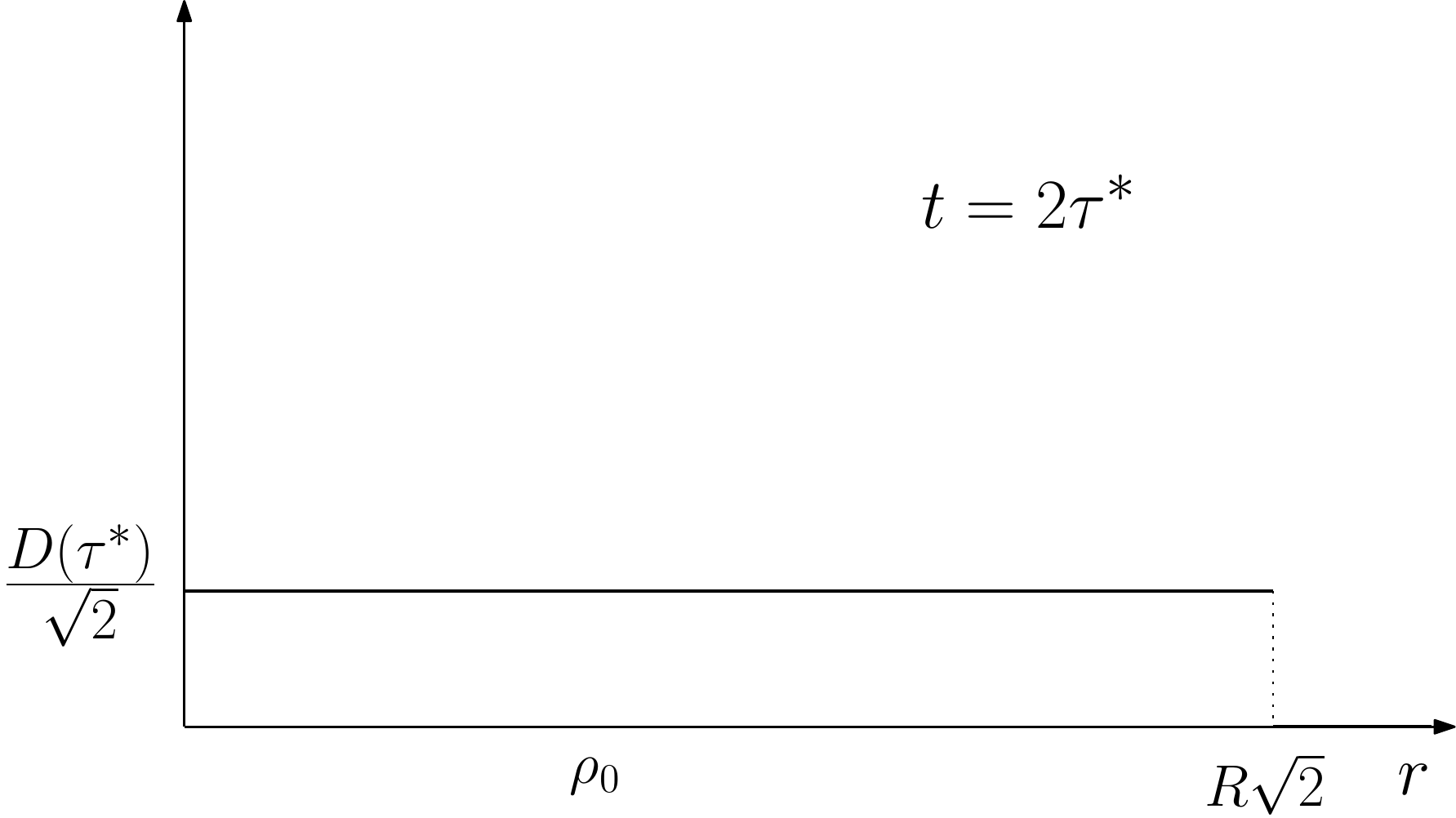} \end{minipage}
\caption{\label{figura2} {The radial profile of the function} $u(t,x)=\tilde{u}(t,r)$ in the case $x_0=0$, $N=1$, $m=2$,  evaluated resp., at $t=0, t=\frac{\tau^*}{2}, \tau^*$ and $2\tau^*$.}
\end{figure}

\begin{remark}
Note that the solutions {constructed in Proposition \ref{prop:wts} are continuous until $\rm \tau^*$, that is, as long as their support does not expand, and develop a jump discontinuity at the boundary of their support at $t=\rm \tau^*$ , that is, as  soon as their support starts expanding (see Figure \ref{figura2} as an example).}
%
%
We believe that such behavior is generic, {in the sense that the support of solutions to \eqref{pbparab-C} expands if and only if a jump discontinuity exists continuous across the support's boundary}.  In next section (see Example \ref{es2}) we will show that singularities may form also in the bulk of the solutions' support, a fact which has been numerically observed (\cite{CCM_plms13}, \cite{ACMSV_siam12}) and analytically shown for some analogous equations in pioneering papers \cite{bl2, bdp}.
\end{remark}


\begin{remark}
{In contrast {to} the case $m=1$, the operator $\mathcal A_g$ is not completely accretive.}

If it were, it would be accretive in $L^p(\Omega)$ for all $1\leq p\leq\infty$ \cite{BenilanCrandallCA}. {In particular, for any $u_0\in{\mathcal D (A_g)}^{L^\infty(\Omega)}$} 
the approximating solution $u_k$ constructed by Crandall-Ligget's scheme \eqref{eulerimplicit} would converge to the mild solution $u(t)=S(t)u_0$ uniformly in time in the $L^\infty(\Omega)$- topology. Therefore, since $u_k(t)\in DTBV_+(\Omega)$ for a.e. $t\in [0,\tau]$ and the convergence is uniform, it would follow that $u(t)\in DTBV_+(\Omega)$, too. {We will now show that this is not the case.}

{Let $u$ and $\tau^*$ as in Proposition \ref{prop:wts}.} We claim that $u_0\in \mathcal D( A_g)$ if $B(0,R)\Subset\Omega$, $g=0$, and $\tau^*\geq\frac{1}{m-1}$. For this, it suffices to prove that $u_0-(\dive \z)(0)\in L^\infty_+(\Omega),$ with $\z$ the vector field defined in the proof of Proposition \ref{prop:wts}, since the other conditions are guaranteed by construction. This is equivalent to show that $$u_0-(u_t)\res_{t=0}\in L^\infty_+(\Omega)\Leftrightarrow\left\{\begin{array}{c} D_0\geq D'(0) \\ C_0\geq C'(0)\end{array}\right.$$ The first inequality is always satisfied while the second one is equivalent to $\tau^*\geq\frac{1}{m-1}$. {Therefore $u_0\in \mathcal D(A_g)$, but the corresponding solution $u(t)\notin DTBV_+(\Omega)$ for $t\geq \tau^*$ and until the time in which $\supp u$ reaches $\partial\Omega$. This contradicts the previous argument, thus proving that $\mathcal A_g$ is not completely accretive.}
\end{remark}

\subsection{Optimal waiting-time bounds}

{The {\it waiting time} is a positive time during which the solution's support, locally in space, does not expand, e.g. $\tau^*$ is the waiting time for the solutions constructed in Proposition \ref{prop:wts}. It is well-known that waiting time phenomena are expected to occur for degenerate parabolic equations, depending on the local behavior of the initial datum. In the next two theorems we provide a scaling-wise sharp condition on the initial datum for the existence of a positive waiting time.}

\begin{thm}\label{wt-lower}
Let ${0\leq} u_{0}\in L^\infty_{loc}(\R^N)\cap L^1(\R^N)$ and let $u$ be the entropy solution to {\eqref{pbparab-C} in $(0,+\infty)\times\R^N$.}
If $x_0\in \R^N$ is such that
$$
\sup_{x\in\R^N} |x-x_0|^{-1/(m-1)} u_0(x) =: L<+\infty,
$$
then
$$
u(t,x_0)=0 \quad \mbox{for all $t<{\rm \tau}_{\rm low}:= \frac{1}{N(m-1)+1}L^{1-m}$}.
$$
\end{thm}

\begin{proof}
We may assume without loss of generality that $x_0=0$. A straightforward computation shows that
$$
\ou(t,x)=\left(\frac{|x|}{(N(m-1)+1)({\rm \tau_{low}}-t)}\right)^{1/(m-1)}
$$
is a solution {to \eqref{pbparab-C} in $(0,{\rm \tau_{low}})\times \R^N$}. In view of the definition of ${\rm  \tau_{ low}}$, we have
$$
u(0,x) \le L |x|^{1/(m-1)} \le \ou(0,x) \quad \mbox{for all $x\in \R^N$},
$$
hence Theorem \ref{T-sub-C} (applied with $\ou$ as solution and $u$ as subsolution) implies that $u(t,x)\le\ou (t,x)$ for $t<{\rm \tau_{low}}$.
\end{proof}

\begin{thm}
Let $0\leq u_{0}\in L^\infty_{loc}(\R^N)\cap L^1(\R^N)$ {be nonnegative}, let $u$ be the entropy solution to {\eqref{pbparab-C} in $(0,+\infty)\times\R^N$, and let $x_0\in \overline{\R^N\setminus \supp(u_0)}$. Let}
$$
t_*=\sup\left\{t\ge 0: \ x_0\in \overline{\R^N\setminus \supp(u(\tau,\cdot))} \quad \mbox{for all $\tau\in [0,t]$}\right\}.
$$
If
\begin{equation}\label{crit-m}
\lim_{\rho\to 0^{+}} \essinf_{x\in B(x_0+\rho \nu_0,\rho)} u_{0}(x) |x-x_0|^{-\frac{1}{m-1}} = \ell\in (0,+\infty],
\end{equation}
for some $\nu_0 \in \mathbb S^{N-1}$, then
\begin{equation}\label{def-W}
t_*\le {\rm \tau}_{\rm up}:= \frac{1}{N(m-1)+1} \ell^{1-m}.
\end{equation}
In particular, $t_*=0$ if $\ell=+\infty$.
\end{thm}

\begin{remark} {Note that the balls in \eqref{crit-m} are nested, hence the infimum with respect to $\rho$ is monotone increasing: therefore the limit in \eqref{crit-m} exists and coincides with the supremum over $\rho$.}
In view of Theorem \ref{wt-lower}, we have
$$
L^{1-m}\le (N(m-1)+1)t_* \le \ell^{1-m}.
$$
Hence the estimate is scaling-wise sharp.
\end{remark}

\begin{proof}
We may assume without loss of generality that $x_0=0$. In view of \eqref{crit-m}, for any $\eps\in(0,\ell)$ there exists $R>0$  such that
\begin{equation}
\label{lbR-u0}
u_{0}(x) \ge (\ell-\eps) |x|^{\frac{1}{m-1}}\quad\mbox{for all $x\in B(R \nu_0,R)$.}
\end{equation}
We wish to choose initial constants in Proposition \ref{prop:wts} such that
$$
\uu(t,x)= {D}(t) \chi_{B_{\rho(t)}} + C(t)\psi(r)\chi_{B_R\setminus B_{\rho(t)}}, \quad B_\rho:=B(R\nu_0,\rho), \quad r={\blue\|}x-R\nu_0\|
$$
 is a {solution with initial datum $\uu(0)\le u_0$, so that we can use it as a subsolution}. {Take $\rho_0< R$.} On $B_{\rho_0}$ we need
\begin{equation}\label{cond1}
\inf_{x\in B_{\rho_0}} u_0(x)  {\stackrel{\eqref{lbR-u0}}\ge} \inf_{x\in B_{\rho_0}} (\ell-\eps) |x|^{1/(m-1)} = (\ell-\eps) |R-\rho_0|^{1/(m-1)} \ge D_0.
\end{equation}
On $B_R\setminus B_{\rho_0}$, for any $r\in [\rho_0,R]$ we need
$$
\inf_{\|x-R\nu_0\|=r} u_0(x) {\stackrel{\eqref{lbR-u0}}\ge} \inf_{|x-R\nu_0|=r} (\ell-\eps) |x|^{1/(m-1)} = (\ell-\eps) |R-r|^{1/(m-1)} \ge C_0\psi(r),
$$
that is,
$$
\left(\ell-\eps\right)^{m-1} |R-r| \ge C_0^{m-1}(\psi(r))^{m-1}= D_0^{m-1} \frac{r}{\rho_0} \frac{\left(\left(\frac{R}{r}\right)^p-1\right)}{ \left(\left(\frac{R}{\rho_0}\right)^p-1\right)},
$$
which is implied by
$$
D_0^{m-1}\le \left(\ell-\eps\right)^{m-1} \rho_0 \left(\left(\tfrac{R}{\rho_0}\right)^p-1\right)\min_{r\in [\rho_0,R]} \frac{ R-r}{r\left(\left(\frac{R}{r}\right)^p-1\right)}.
$$
It is easy to see that the function to be minimized is increasing (take $x=r/R\in (0,1)$). Hence the minimum is attained at $r=\rho_0$, so that we need
$$
D_0^{m-1} \le \left(\ell-\eps\right)^{m-1} (R-\rho_0),
$$
which coincides with \eqref{cond1}. We choose equality. Therefore {$\uu$ is a subsolution, hence $\uu\le u$ by Theorem \ref{T-sub-C}.} Since the support of $\uu$ starts expanding at time ${\rm \tau^*}$, we have
$$
t_*\le {\rm \tau^{\blue *}}=\frac{\rho_0 A_0^{1-m}} {{mp}}\left(\left(\frac{R}{\rho_0}\right)^p-1\right) = \left(\ell-\eps\right)^{1-m}\frac{1}{mp} \frac{R^p-\rho_0^p}{\rho_0^{p-1}(R-\rho_0)}
$$
for all $\rho_0\in (0,R)$. Minimizing with respect to $\rho_0$ and recalling the arbitrariness of $\eps$ and the definition of $p$ yields the conclusion.
\end{proof}

\section{Burgers' type dynamics}

In this section, we concentrate on the one-dimensional case:
\begin{equation}
\label{eq-1d}
u_{t}=  \left(u^{m}\frac{u_{x}}{| u_{x}|}\right)_{x}.
\end{equation}
Formally speaking, $\frac{u_{x}}{| u_{x}|}$ is constant on intervals in which $u$ is strictly monotone, whence \eqref{eq-1d} reduces to a nonlinear conservation law: for instance,
$$
u_{t}=  \left(u^{m}\frac{u_{x}}{| u_{x}|}\right)_{x}=- \left(u^{m}\right)_{x} \quad \mbox{in $J\times I$} \quad\mbox{if $u(t,\cdot)$ is decreasing in $I$ for a.e. $t\in J$}.
$$
This formal observation suggests that the behavior of solutions to \eqref{eq-1d} is strictly related to that of {\blue a} nonlinear conservation law. In what follows we give two examples of the relationship between the two: in the first one, solutions in fact coincide; in the second one, instead, the qualitative and quantitative properties turn out to differ sensibly.

\smallskip

{
Prior to the examples, let us recall that an entropy solution to the generalized Burgers equation,}
\begin{equation}
\begin{cases}
v_{t}= -  \left(v^{m}\right)_{x} &  \text{in}\,  (0,{\rm \tau})\times\re, \\
v(0)=v_{0} & \text{in}\  \re\,
\end{cases}\label{pdeb}
\end{equation}
with $m>1$, is a bounded function $v\in L^\infty((0,\infty); TBV_{loc}(\R))$ satisfying \eqref{pdeb}$_1$ in distributional sense, $v(0)=v_{0}$,  and
\begin{equation}\label{K1}
\eta(v)_t + (q(v))_x \le 0 \quad\mbox{in $\mathcal D'(\R)$}
\end{equation}
for all convex functions (entropies) $\eta$, with corresponding entropy flux $q$ defined by $q'(v) = mv^{m-1}\eta'(v)$ (see e.g. \cite{LOW}).

\begin{example}
Let $\Omega=]0,R[$, let $u_0:\Omega\to \R$ be nonincreasing. Assume that ${\rm supp} (u_0)\subset [0,R[$. Then the {entropy} solution to
\begin{equation}
\begin{cases}
u_{t}=  \left(u^{m}\frac{u_{x}}{| u_{x}|}\right)_{x} &  \text{in}\,  (0,{\rm \tau})\times [0,R], \\
u(0)=u_{0} & \text{in}\  [0,R], \\ u(t,0)= u_0(0), \ u(t,R)=0& \text{for}\ t>0,
\end{cases}\label{pde}
\end{equation}
coincides in $[0,R]$ with the entropy solution $v$ to \eqref{pdeb} with

\begin{equation}\label{pdeb0}
v_0(x)= \left\{\begin{array}{ll} u_0(0) & \mbox{if $x\le 0$}
\\ u_0(x) & \mbox{if $x\in [0,R]$}
\\ 0 & \mbox{if $x\ge R$}.
\end{array}
\right.
\end{equation}
\end{example}

\begin{proof}
{\blue Let $v$ be the entropy solution to \eqref{pdeb} with \eqref{pdeb0} as initial datum. We will show that $u:=v\lfloor_{[0,R]}$ is a solution to \eqref{pde}.}

{It follows from \eqref{pdeb0}, the monotonicity of $u_0$, and} Lax-Oleinik formula (see e.g. \cite{evans}) that
\begin{equation}\label{K2}
\mbox{$v(t,x)=u_0(0)\ $ for all $t>0$ and all $x<m (u_0(0))^{m-1}t$}
\end{equation}
and
\begin{equation}
\label{K3}
\mbox{$v_x(t,{\cdot})\le 0 \ $ as a measure in $\R$ {for all $t>0$}.}
\end{equation}
Choosing $\eta(v)=J_\ell(v)$ with $\ell\in\mathcal L$, we have $\eta'(v)=\ell(v)$,
$$
q_\ell(v)=\int_0^{v} mw^{m-1}\ell(w)\d w = v^m \ell(v)- \int_0^v v^m \ell'(w)\d w = v^m \ell(v)- {\blue \Phi}_\ell(v),
$$
and for any nonnegative $\psi\in C^\infty_c((0,\infty)\times \R)$ it holds that
\begin{eqnarray*}
\iint_{(0,\infty)\times \R} J_\ell(v)\psi_t &\ge & - \iint_{(0,\infty)\times \R} \left(v^m \ell(v)- \Phi_\ell(v)\right)\psi_x
\\ &=&  - \iint_{(0,\infty)\times \R} v^m \ell(v)\psi_x - \iint_{(0,\infty)\times \R} \left(\Phi_\ell(v)\right)_x\psi
\\ &\stackrel{\eqref{K3}} =& - \iint_{(0,\infty)\times \R} v^m \ell(v)\psi_x + \iint_{(0,\infty)\times \R}\psi\d\left|\left(\Phi_\ell(v)\right)_x\right|,
\end{eqnarray*}
whence \eqref{main-ineq} holds choosing $z= -v^m$. Condition \eqref{boundconddu>g} is immediate from \eqref{K2} 
{\blue and the fact that $v$ is nonnegative}. Condition \eqref{dist} follows from \eqref{pdeb}$_1$ and the choice of $z$; the regularity $\ell(v)\in L^1([0,{\rm \tau});BV (0,R))$  for all $\ell\in \mathcal L$ follows from the regularity of $v$. Observe that the boundary condition \eqref{boundcondd=} is automatically satisfied. Hence, since $v \in C([0,{\rm \tau});L^1(0,R))$ and  $v_t\in L^\infty_{loc}((0,{\rm \tau}],\mathcal M(0,R))$ (see, for instance \cite{BenilanCrandallCA}), the proof is finished.
\end{proof}

The next example shows that instead, for the Cauchy problem, the solution's behavior is different from that of the associated Burgers equation.

\begin{example}\label{es2}
  Let
  $$
  u_0(x)=2\chi_{[0,1]}+(3-x)\chi_{[1,2]}+\frac{9-x}{7}\chi_{[2,9]} \quad\mbox{for $x\ge 0$}, \qquad u_0(x)=u_0(-x) \quad\mbox{for $x\le 0$}.
  $$
\begin{figure}[htbp]\centering
\includegraphics[keepaspectratio,height=3cm, width=6cm]{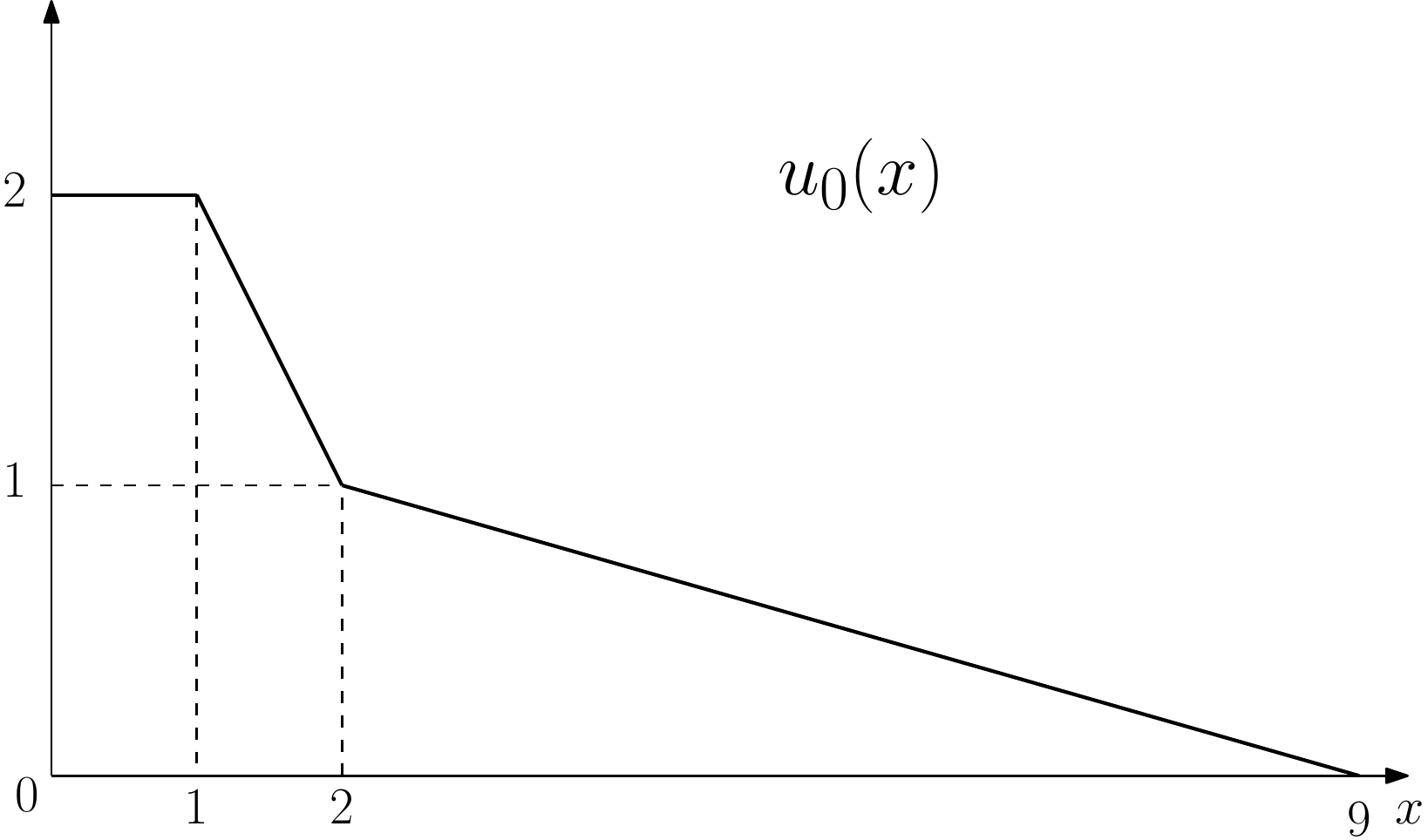}
\caption{$u_0(x)$}
\end{figure}

\noindent Then there exist $t_*\in(\frac{1}{2},\frac{7}{2})$ and nonnegative functions ${\blue D},r\in C([0,t_*-\frac{1}{2}])$ with $D$ decreasing, $D(t-\frac{1}{2})>\frac{9-x}{7-2t}$ in $(\frac{1}{2},t_*)$ and $r$ increasing with $r(0)=3$ and $r(t_*-\frac{1}{2})<9$, such that the solution to
\begin{equation}
\begin{cases}
u_{t}=  \left(u^{2}\frac{u_{x}}{| u_{x}|}\right)_{x} &  \text{in}\,  (0,{\rm \tau})\times \re, \\
u(0)=u_{0} & \text{in}\  \re
\end{cases}\label{pdebur}
\end{equation}
is symmetric with respect to $x=0$ and for $x\ge 0$ is given by:
\begin{equation}
\label{ex2-sol}
u(t,x)=\left\{\begin{array}{ll} \dys \frac{3-\sqrt{16t+1}}{1-2t}\chi_{[0,\sqrt{16t+1}]}(x)+\frac{3-x}{1-2t} \chi_{[\sqrt{16t+1},{2+2t}]}(x)+\frac{9-x}{7-2t}\chi_{[{2+2t},9]}(x)  & \quad\mbox{if $t<1/2$,}
\\[3ex] \dys
D\left(t-\frac{1}{2}\right)\chi_{\left[0,r\left(t-\frac{1}{2}\right)\right)} +\frac{9-x}{7-2t}\chi_{\left(r\left(t-\frac{1}{2}\right),9\right]} & \quad \mbox{if $\frac{1}{2}\le t<t_*$}
\\[3ex] \dys \frac{9-\sqrt{28t-17}}{7-2t}\chi_{[0,\sqrt{28t-17}]}+\frac{9-x}{7-2t}\chi_{[\sqrt{28t-17},9]} & \quad t_*\leq t\leq \frac{7}{2}
\\[3ex] \dys \left(\frac{32}{49}+\frac{2 t}{7}\right)^{\frac{-1}{2}}\chi_{[0,7\sqrt{\frac{32}{49}+\frac{2t}{7}}]} & t\geq \frac{7}{2} \end{array} ,\right.
\end{equation}

\begin{figure}[H]
\begin{minipage}[c]{0.33\textwidth}
\includegraphics[keepaspectratio,height=5.5cm, width=5.5cm]{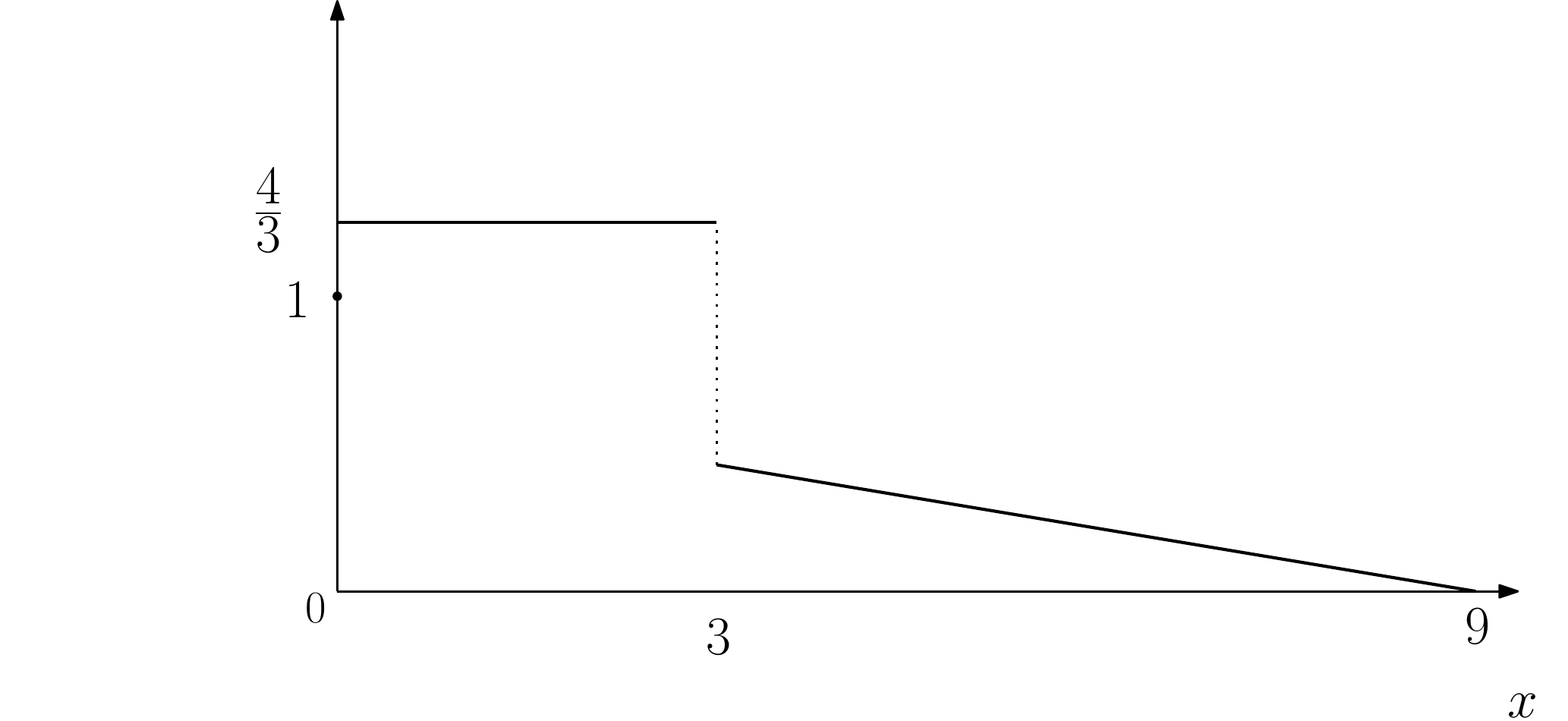}\end{minipage}
\begin{minipage}[c]{0.33\textwidth}\includegraphics[keepaspectratio,height=5.5cm, width=5.5cm]{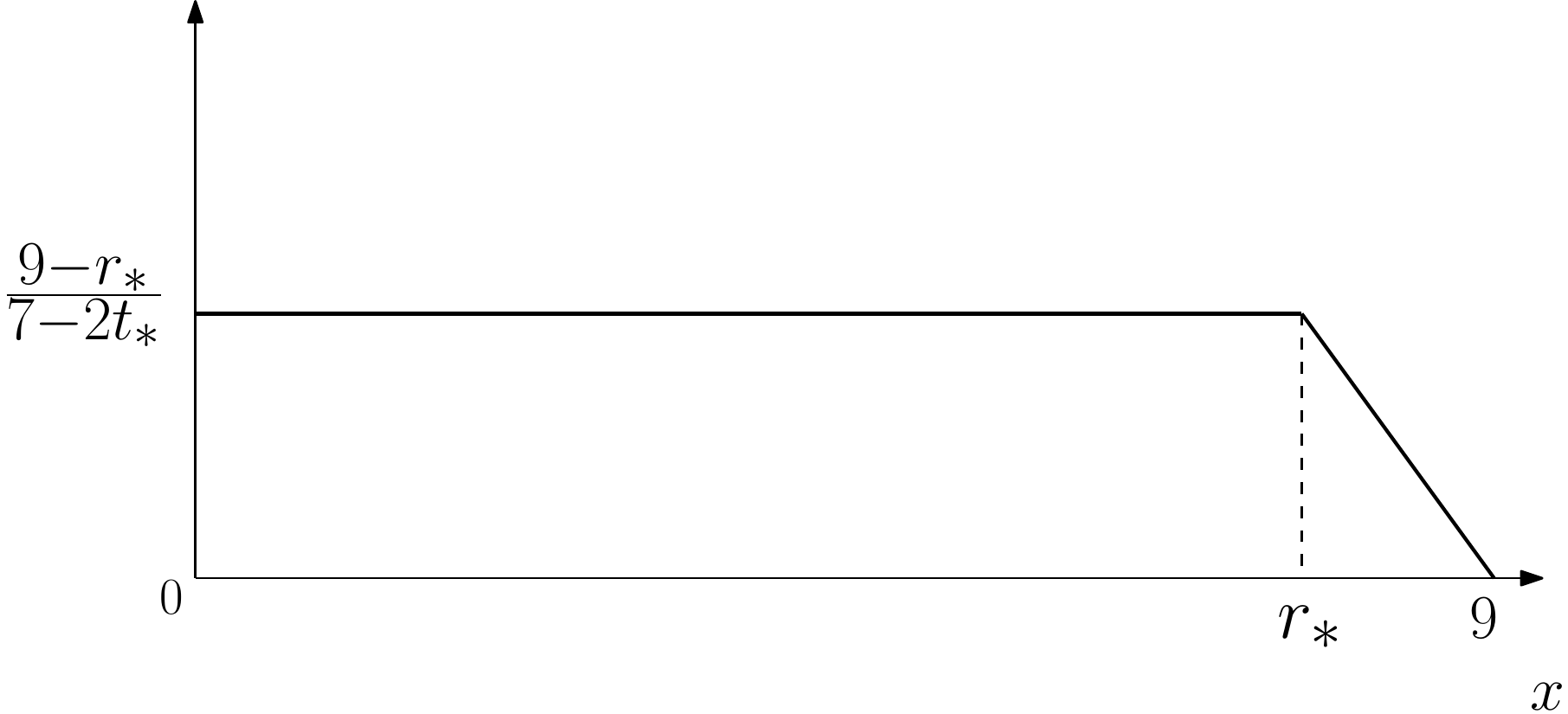} \end{minipage}
\begin{minipage}[c]{0.33\textwidth}\includegraphics[keepaspectratio,height=5.5cm, width=5.5cm]{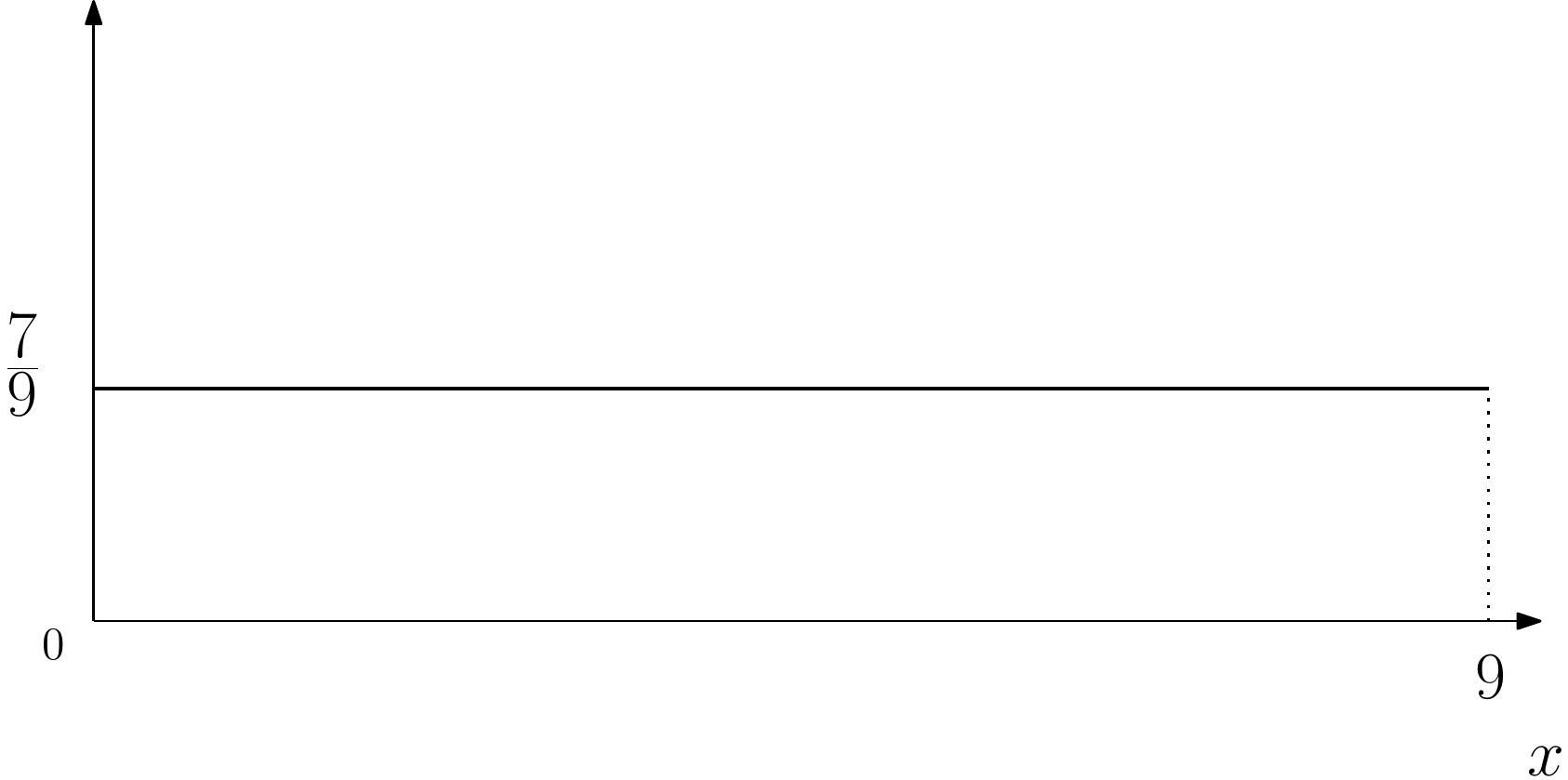} \end{minipage}
\caption{{The function} $u$ {in \eqref{ex2-sol}} at times $t_1=\frac12$, $t_2=t_*$, and $t_3={\frac72}$.}
\end{figure}

\end{example}

Before the proof, let us briefly comment on the structure of such solution, also by comparing it with the solution to the Burgers equation
\begin{equation}
\begin{cases}
v_{t}= -  \left(v^{2}\right)_{x} &  \text{in}\,  (0,{\rm \tau})\times\re, \\
v(0,x)=\left\{\begin{array}{ll} u_{0}(x) & \text{for}\  x\ge 0 \\ u_{0}(0) & \text{for}\  x\le 0,
\end{array}\right.
\end{cases}\label{pdebburg}
\end{equation}
which can be easily found by the method of characteristics:
\begin{equation}\label{sol-burg}
v(t,x)=\left\{\begin{array}{ll} \displaystyle 2\chi_{(-\infty,1+4t]}(x) + \frac{3-x}{1-2t}\chi_{(1+4t,2+2t]}(x) + \frac{9-x}{7-2t} \chi_{(2+2t,9]}(x)
 &\mbox{if $t<1/2$}
 \\[2ex] \displaystyle 2\chi_{(-\infty,r_v(t)]}(x) + \frac{9-x}{7-2t} \chi_{(r_v(t),9]}(x) & \mbox{if $1/2\le t<11/4$},
 \\[2ex] \displaystyle 2\chi_{(-\infty,9+2(t-11/4)]}(x) &  \mbox{if $t>11/4$},
\end{array}\right.
\end{equation}
where $r_v(t)= \sqrt{42-12t} + 4t-5$.

\smallskip

\begin{figure}[H]
\begin{minipage}[c]{0.33\textwidth}
\includegraphics[keepaspectratio,height=5.5cm, width=5.5cm]{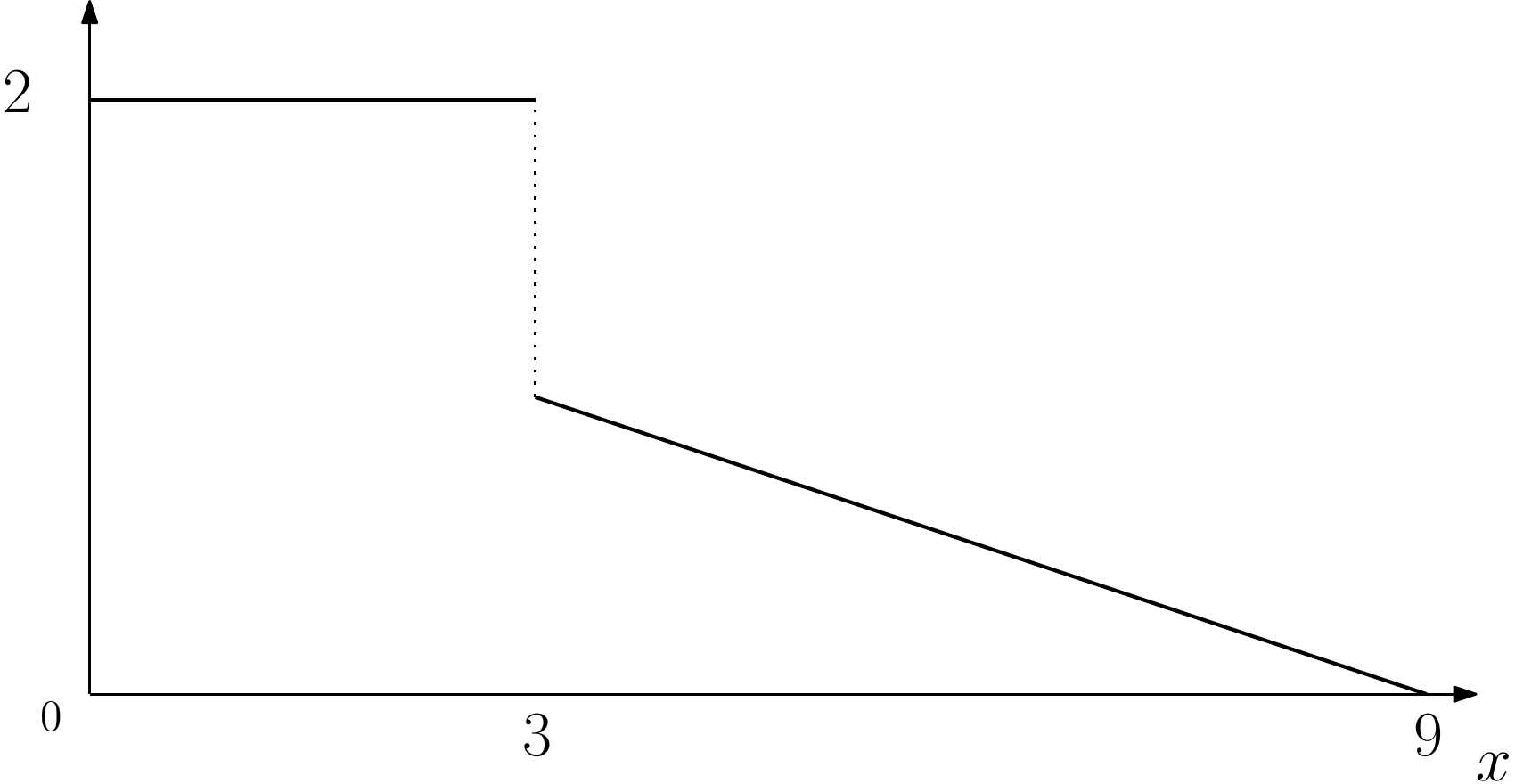}\end{minipage}
\begin{minipage}[c]{0.33\textwidth}\includegraphics[keepaspectratio,height=5.5cm, width=5.5cm]{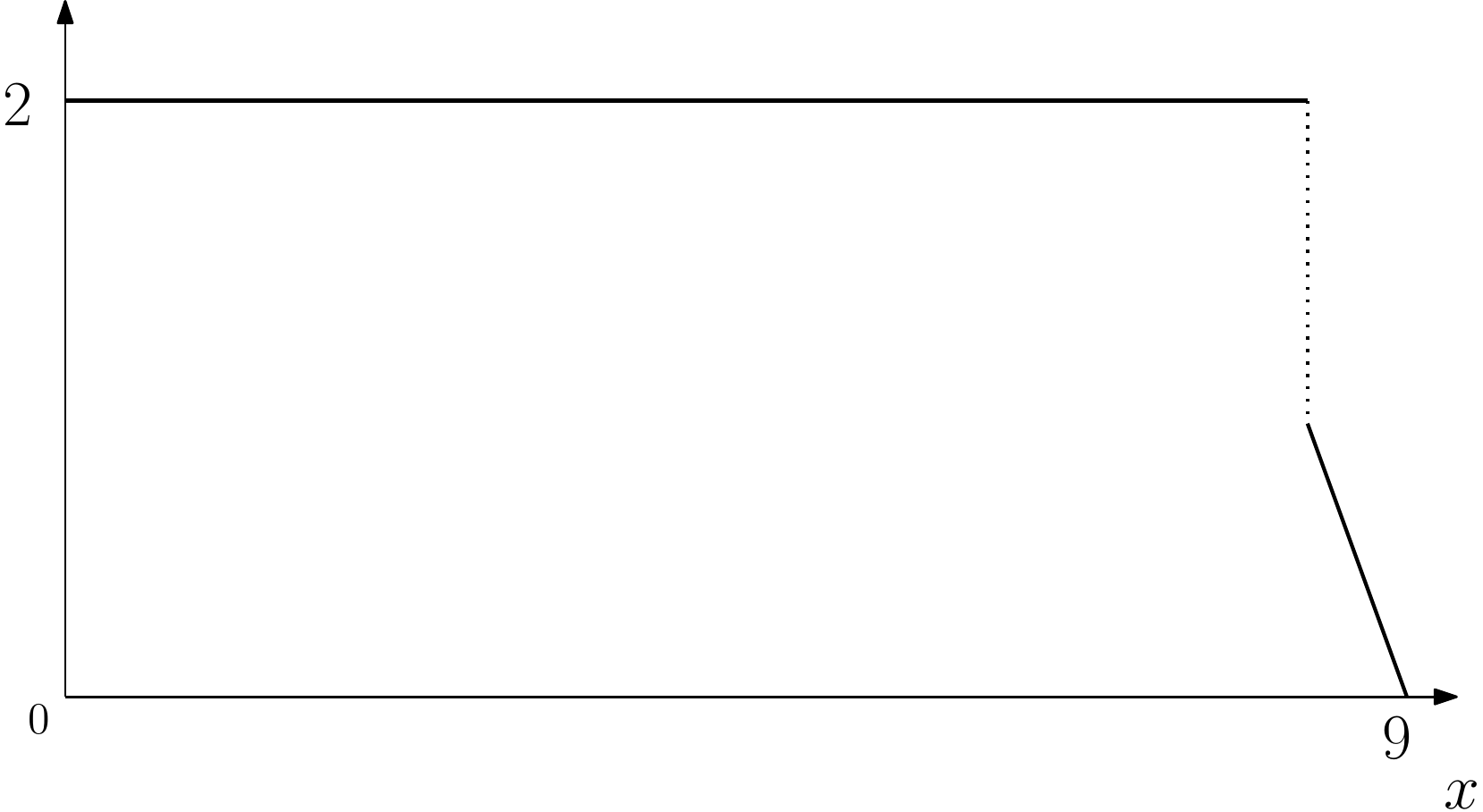} \end{minipage}
\begin{minipage}[c]{0.33\textwidth}\includegraphics[keepaspectratio,height=5.5cm, width=5.5cm]{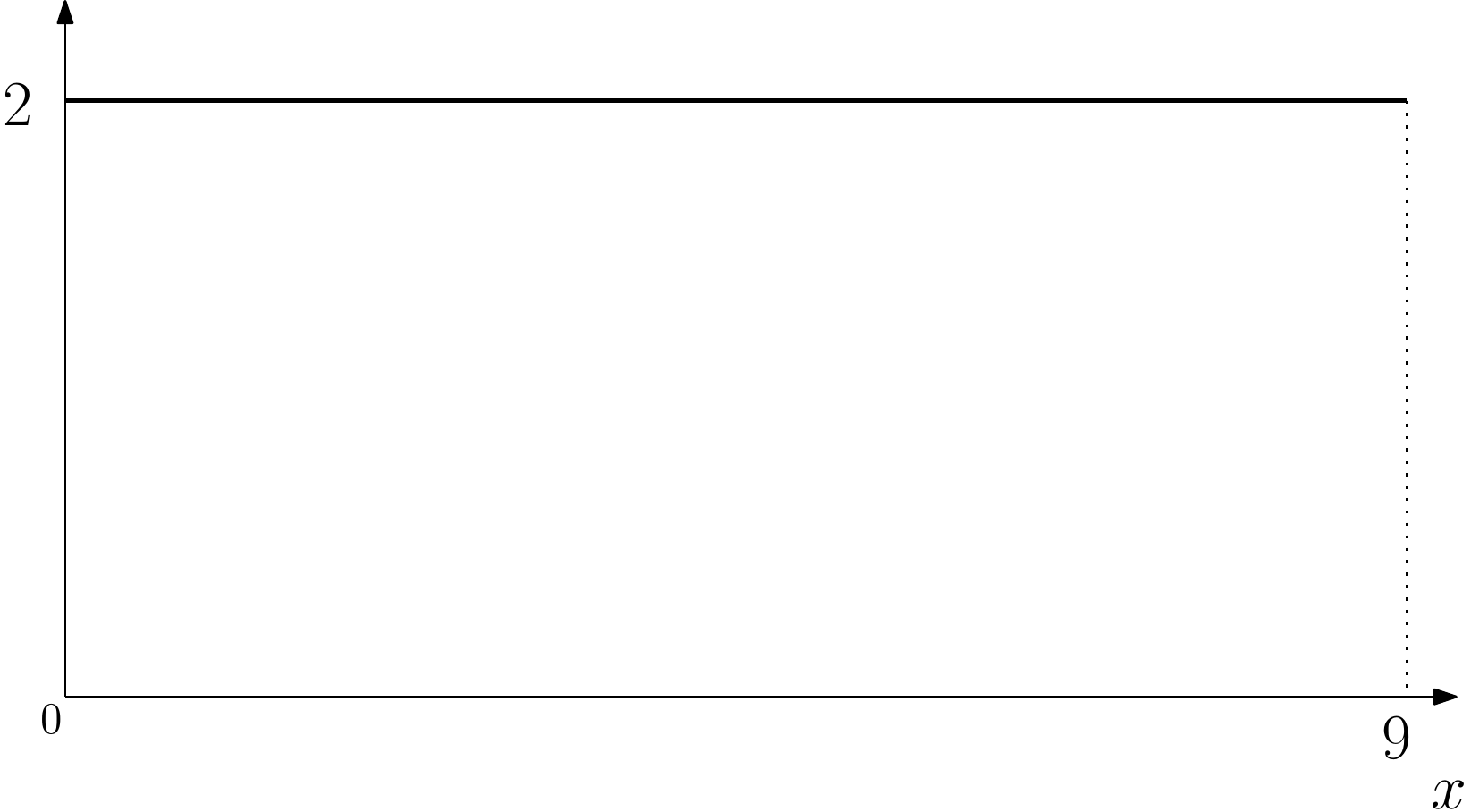} \end{minipage}
\caption{{The function} $v$ {in \eqref{sol-burg}} at times $t_1=\frac12$, $t_2\in \left(\frac12,\frac{11}{4}\right)$, and $t_3=\frac{11}{4}$}
\end{figure}

{The behaviour of $u$ and $v$ for $x\le 0$ is obviously different ($u$ is even, $v$ is constant for $x\le 0$) and does not deserve comments. Comparing $u$ and $v$ for $x\ge 0$, two different features should be noted. Firstly, the bulk singularity (which is formed in both cases at $t_1=\frac12$) persist for $v$, whereas it vanishes at time $t_*$ for $u$. Hence (by the Rankine-Hugoniot condition, which holds in both cases) the bulk singularity travels faster for $v$ than for $u$ ($\frac{11}{4}<\frac72$). Secondly, the height of the plateau  is constant for $v$, whereas it decreases for $u$. The nonlocal effect caused by mass constraint is the source of both of these qualitative differences.
}

\begin{proof}
\ignore{Let us first describe the construction of the candidate solution $u$ in \eqref{ex2-sol}. 
Observe that the candidate solution is symmetrically nonincreasing. Once the construction is finished, it suffices to take
$$
z(t,x):=\left\{\begin{array}{ll} \dys -\frac{x}{\sqrt{16t+1}}\chi_{[0,\sqrt{16t+1}]}-\chi_{[\sqrt{16t+1},+\infty)} & {\rm if \ } 0\leq t\leq \frac{1}{2}
\\[2ex] \dys -\frac{x}{r(t-\frac{1}{2})}\chi_{[0,r(t-\frac{1}{2})]}-\chi_{[r(t-\frac{1}{2}),+\infty)} & {\rm if \ } \frac{1}{2}<t\leq t^*
\\[2ex] \dys -\frac{x}{\sqrt{28t-17}}\chi_{[0,\sqrt{28t-17}]}-\chi_{[\sqrt{28t-17},+\infty)} & {\rm if \ } t^*<t\leq \frac{7}{2}
\\[2ex] \dys -\frac{x}{\sqrt{28t-17}}\chi_{[0,\sqrt{28t-17}]}-\chi_{[\sqrt{28t-17},+\infty)} & {\rm if \ } t^*<t\leq \frac{7}{2}
\\[2ex] \dys -\frac{x}{7\sqrt{\frac{32}{49}+\frac{2t}{7}}}\chi_{[0,7\sqrt{\frac{32}{49}+\frac{2t}{7}}]}-\chi_{]7\sqrt{\frac{32}{49}+\frac{2t}{7}}]} & {\rm if \ }t>\frac{7}{2}
\end{array}\right.,
$$
and symmetrize it. Then, it is immediate to show that $u_t= z_x$ and that \eqref{conditionatjump} hold. Since $u_x^c=0$, by Theorem \ref{thm:charac}, $u$ will be the solution to \eqref{pdebur}.
}
Of course $u$ will be symmetric with respect to $x=0$, hence we only work for $x\ge 0$. The candidate solution $u$ is constructed as follows: as long as the first singularity of $v$ appears, $u$ behaves as $v$, except for the fact that mass needs to be preserved: hence the flat region on top expands and decreases: for $x\ge 0$,
$$
u(t,x)= D_1(t)\chi_{[0,r_1(t)]}(x) + \frac{3-x}{1-2t}\chi_{(r_1(t),2+2t]}(x) + \frac{9-x}{7-2t} \chi_{(2+2t,9]}(x) \quad \mbox{if $t<1/2$},
$$
where $D_1$ and $r_1$ have to be obtained by imposing continuity of $u$ and mass conservation, that is,
$$
D_1(t)=\frac{3-r_1(t)}{1-2t}, \quad\mbox{resp.} \quad 7= D_1(t)r_1(t)+ \int_{r_1(t)}^{2+2t}\frac{3-x}{1-2t}\d x + \int_{2+2t}^9\frac{9-x}{7-2t}\d x.
$$
Solving the equation gives $D_1$ and $r_1$ as in \eqref{ex2-sol}. At $t=1/2$,
$$
u(1/2,x)=\frac{4}{3}\chi_{[0,3]}(x)+\frac{9-x}{6}\chi_{]3,9]}.
$$

We now consider $s:=t-1/2>0$. Then
$$
u(s+1/2,x)=D(s)\chi_{[0,r(s)[}+\frac{9-x}{2(3-s)}\chi_{]r(s),9]}.
$$
In this case, we recover $D$ and $r$ from the Rankine-Hugoniot condition and mass conservation, that is,
\begin{equation}\label{resp21}
r'(s)=D(s)+\frac{9-r(s)}{2(3-s)},
\end{equation}
respectively
\begin{equation}\label{resp22}
7= D(s)r(s) + \int_{r(s)}^9\frac{9-x}{2(3-s)}\d x=D(s)r(s) + \frac{(9-r(s))^2}{4(3-s)},
\end{equation}
as long as $s<3$ and
$$
0<C(s):= u^+(s +1/2,r(s))=\frac{9-r(s)}{2(3-s)} <D(s).
$$
We now argue for $s<3$. As long as it is defined, $C$ solves
\begin{eqnarray}
\label{edoB1}
C'(s) &=& \frac{C(s)-D(s)}{2(3-s)}
\\
\label{edoB2}
 &= & \frac{-(3-s)C^2(s)+9C(s)-7}{2(3-s)(9-2(3-s)C(s))}
\end{eqnarray}
with initial condition $C(0)=1$. Since $B$ is initially decreasing and
$$
9-2(3-s)C(s))>0 \ \iff \ C<\frac32 <\frac{9}{2(3-s)},
$$
$C$ is well defined as long as $C<3/2$. Equation \eqref{edoB2} may be integrated implicitly, yielding
$$
\arctanh\left(\sqrt{\frac{3 - s}{7}} C(s)\right) -
\frac{ \sqrt{7(3 - s)} (14 - 9 C(s))}{63 - 9 (3 - s) B^2(s)} =
 \arctanh\left(\sqrt{\frac37}\right) - \frac{5\sqrt{21}}{36}
$$
Therefore
$$
C(s)=1 \ \iff \ f(s):=\arctanh\left(\sqrt{\frac{3 - s}{7}}\right) -\arctanh\left(\sqrt{\frac37}\right)
-\frac{ 5\sqrt{7(3 - s)}}{36 + 9s}
+ \frac{5\sqrt{21}}{36}=0.
$$
We already know that $f(0)=0$. Simple computations show that $f'(0)=\frac{7\sqrt{7}}{144\sqrt{3}}>0$ and
$$
f(3)= \frac{5\sqrt{21}}{36}-\arctanh\left(\sqrt{\frac37}\right)<0\,.
$$
 Therefore there exists $s_1\in (0,3)$ such that $f(s_1)=0$, i.e. $C(s_1)=1=C(0)$. By Rolle's theorem, there exists $s_*\in (0,s_1)$ such that $C'(s_*)=0$ and $C'<0$ in $(0,s_*)$. Then \eqref{edoB1} implies that $D(s_*)=C(s_*)=\frac{9-r(s_*)}{2(3-s_*)}$. Hence $r_*=r(s_*)<9$, and it follows from mass conservation (Propostition \ref{prop-mass-cons}) that
$$
7= D(s_*)r_* + \frac{(9-r_*)^2}{4(3-s_*)}=\frac{(81-r_*^2)}{2(6-2s_*)},
$$
whence $r_*:=r(s_*)<9$. This completes the construction of \eqref{ex2-sol} in the time interval $[\frac{1}{2},t^*]$. At $t=t_*=s_*+1/2<7/2$, we have
$$
u(t_*,x) =\frac{9-r_*}{7-2t_*}\chi_{[0,r_*)[}(x)+\frac{9-x}{7-2t_*}\chi_{[r_*,9]}(x), \quad r_*=\sqrt{81 - 14(7-2t_*)}=\sqrt{28t_*-17},
$$
a continuous (piecewise linear) function. Hence we may argue as in the construction of the solution in the time interval $[0,\frac{1}{2}]$, obtaining
$$
u(t,x)=\frac{9-\sqrt{28t-17}}{7-2t}\chi_{[0,\sqrt{28t-17}]}+\frac{9-x}{7-2t}\chi_{[\sqrt{28t-17},9]}\,\quad t_*\leq t\leq \frac{7}{2}.
$$
At $t_{**}=\frac{7}{2}$, the solution develops a new singularity, since $\sqrt{28t-17}\to 9$ and  $\frac{9-\sqrt{28t-17}}{7-2t}\to \frac79$ as $t\to t_{**}^-$. Therefore
$$
u(t_{**})=\frac{7}{9}\chi_{[0,9]}.
$$
After $t_{**}$ the solution becomes the self-simliar one obtained in Theorem \ref{th:ss}; i.e.
$$
u(t,x)=\left(\frac{32}{49}+\frac{2t}{7}\right)^{\frac{-1}{2}}\chi_{[0,7\sqrt{\frac{32}{49}+\frac{2t}{7}}]}.
$$


\end{proof}


\acks The second author acknowledges partial support by  Spanish MCIU
and FEDER project PGC2018-094775-B-I00. The authors also acknowledge partial support by GNAMPA of the italian  Instituto Nazionale di Alta Matematica.

\end{document}